%% file: tighter_bound_estimation_for_efficient_biquadratic_optimization.tex
\UseRawInputEncoding
\documentclass[review,onefignum,onetabnum]{siamart190516}
\pdfoutput=1

\input{ex_shared}
\usepackage{soul}
\usepackage{setspace}

\ifpdf
\hypersetup{
  pdftitle={
  },
  pdfauthor={S. Li, L. Lu, X. Qiu, Z. Chen, and D. Zeng}
}
\fi

\begin{document}

\maketitle

\begin{abstract}
Bi-quadratic programming over unit spheres is a fundamental problem in quantum mechanics introduced by pioneer work of Einstein, Schr\"odinger, and others. It has been shown to be NP-hard; so it must be solve by efficient heuristic algorithms such as the block improvement method (BIM). This paper focuses on the maximization of bi-quadratic forms, which leads to a rank-one approximation problem that is equivalent to computing the M-spectral radius and its corresponding eigenvectors. Specifically, we provide a tight upper bound of the M-spectral radius for nonnegative fourth-order partially symmetric (PS) tensors, which can be considered as an approximation of the M-spectral radius. Furthermore, we showed that the proposed upper bound can be obtained more efficiently, if the nonnegative fourth-order PS-tensors is a member of certain monoid semigroups. Furthermore, as an extension of the proposed upper bound, we derive the exact solutions of the M-spectral radius and its corresponding M-eigenvectors for certain classes of fourth-order PS-tensors. Lastly, as an application of the proposed bound, we obtain a practically testable sufficient condition for nonsingular elasticity M-tensors with strong ellipticity condition.

We conduct several numerical experiments to demonstrate the utility of the proposed results. The results show that: (a) our proposed method can attain a tight upper bound of the M-spectral radius with little computational burden, and (b) such tight and efficient upper bounds greatly enhance the convergence speed of the BIM-algorithm, allowing it to be applicable for large-scale problems in applications.

\end{abstract}



\section{Introduction}
The main focus of this study is the following general bi-quadratic polynomial maximization over unit spheres of the form
\begin{align}\label{optimalcondition}
\begin{aligned}
&{\max\limits_{\xx\in \mathbb R^{m},\yy\in \mathbb R^{n}}~
f_{\mathcal A}(\xx,\yy)=\sum_{i,k=1}^{m}\sum_{j,l=1}^{n}a_{ijkl}x_{i}y_{j}x_{k}y_{l},}\\
&{\text{subject~to}\,~~\xx^\top\xx=1,\,\yy^\top \yy = 1,}
\end{aligned}
\end{align}
where $\mathbb R^{n}$ and $\mathbb R^{m}$ denote the Euclidean spaces of dimensions $n$ and $m$;
$x_{i}$ and $y_{j}$ denote the $i$th and $j$th component of $\xx$ and $\yy$, respectively;
and $\mathcal A=(a_{ijkl})\in\mathbb R^{m\times n\times m\times n}$ is a fourth order partially symmetric (\textbf{PS}) tensor such that
$$a_{ijkl}\in\mathbb R, \quad a_{ijkl}=a_{kjil}=a_{ilkj}, \quad \forall i,k\in[m],~\forall j,l\in[n].$$

Here $[p]:=\{1,2,\ldots,p\}$ and $\mathbb R$ is the real field. $\mathcal A$ is further called \textbf{Nonnegative} fourth-order PS-tensor, if all elements in $\mathcal{A}$ are nonnegative, namely, $a_{ijkl}\geq 0$.

The maximization problem~\eqref{optimalcondition} and its minimization counterpart 
have been widely studied in recent years, see \cite{ling2010,hu2010mp,wang2009practical,zhang2011b,Hu2011,yang2012b,yang2014b,wang2015c,qi2019bi-quadratic}
and references therein. They are known to be NP-hard~\cite{ling2010,hu2010mp,wang2009practical}. Problem (\ref{optimalcondition}) arises from the entanglement problem~\cite{doherty2002distinguishing,dahl2007}, which is a fundamental concept in quantum mechanics introduced by the pioneer work of Einstein \cite{einstein}, Schr{\"o}dinger \cite{Schrodinger}, etc. Entanglement is used to describe whether a quantum state is separable or inseparable, or whether a bipartite mixed state $mn\times mn$ can be decomposed as a convex combination of tensor products for pure product states of $n$ and $m$ dimensions \cite{doherty2002distinguishing,dahl2007}. Based on the geometrical description of separability, the problem of finding the closest separable state to any given state can be mathematically transformed into problem (\ref{optimalcondition}) \cite{wang2015c}.  Problem (\ref{optimalcondition}) can also be considered as the matching problem between two images, where the fourth-order PS-tensor $\mathcal A$ can be represented as some attributes among the two images, and the optimal solutions $\xx^{\ast}$ and $\yy^{\ast}$ can be regarded as the matching features of the two images that have the maximum correspondence under the attributes $\mathcal A$ \cite{Leordeanu2005,Cour2007,Egozi2012}. It can also arises from the strong ellipticity condition problem in solid mechanics (for $n = m = 3$) \cite{chirictua2007strong,han2009conditions}. Thus, the fourth-order PS-tensor is also known as the \textbf{Elasticity} tensor \cite{li2019m,wang2018best,miao2020}.

In fact, any fourth-order PS-tensor $\mathcal A$ in problem (\ref{optimalcondition}) can be considered as a linear transformation between the symmetric matrix spaces: ${\mathcal A}: S_{m\times m} \to S_{n\times n}$, such that
\begin{equation*}
  \big(\mathcal A (X)\big)_{jl}=Z_{jl} := \sum_{i,k=1}^{m} a_{ijkl} X_{ik},~~X\in S_{m\times m},~Z\in S_{n\times n},
\end{equation*}
where $S_{p\times p}$ denotes the space consisting of all symmetric matrices with size $p\times p$.
Here, $X$ can be the Laplace matrix or the Toeplitz matrix in graph signals and image processing \cite{dong2016l,yeke2016}.
If $X$ and $Z$ are both
the connectivity matrices in brain image analysis, or unidirectional networks, then the attributes $\mathcal A$ will be a nonnegative fourth-order PS-tensor \cite{Kriegeskorte,FengHypoTest}.
${\mathcal A}$ can also be considered as the representation of a real \textbf{bilinear} function ${\cal A}(X,Y): S_{m\times m} \times S_{n\times n} \to \mathbb R$,
  \begin{equation}
    \label{eq:ET-bilinear-rep1}
    {\cal A}(X,Y) :=\sum_{i,k=1}^{m}\sum_{j,l=1}^{n} a_{ijkl}X_{ik}Y_{jl}.
  \end{equation}

A particularly important special case is when $X=\xx\xx^\top$ and $Y=\yy\yy^\top$ for $\xx\in\mathbb R^{m}$ and $\yy\in\mathbb R^{n}$. According to Equation~\eqref{eq:ET-bilinear-rep1}, we have
  \begin{equation*}
    \label{eq:ET-bilinear-rank1-rep}
    \mathcal{A}(X, Y) = \sum_{i,k=1}^{m}\sum_{j,l=1}^{n} a_{ijkl}X_{ik}Y_{jl} = \sum_{i,k=1}^{m}\sum_{j,l=1}^{n} a_{ijkl}\cdot x_{i}x_{k}y_{j}y_{l} =f_{\mathcal A}(\xx,\yy).
  \end{equation*}
In other words, $f_{\mathcal A}(\xx, \yy)$ represents the special case of bilinear function $\mathcal A(X, Y)$ when the inputs are two rank-1 symmetric matrices. For convenience, we sometimes write $\mathcal{A}(X, Y)$ as $\mathcal A(\xx,\yy,\xx,\yy)$ in this case.

When considering bi-quadratic optimization problem (\ref{optimalcondition}), we usually need to compute the maximum M-eigenvalue and its corresponding eigenvectors of the fourth-order PS-tensor $\mathcal A$ \cite{wang2009practical,wang2015c}. For example, in \cite{han2009conditions,qi2009conditions}, we have definition of
\begin{equation*}
\big(\mathcal A(\cdot, \yy,\xx,\yy)\big)_{i}=\sum_{k=1}^{m}\sum_{j,l=1}^{n}a_{ijkl}y_{j}x_{k}y_{l},~
\big(\mathcal A(\xx,\yy,\xx,\cdot)\big)_{l}
=\sum_{j=1}^{n}\sum_{i,k=1}^{m}a_{ijkl}x_{i}y_{j}x_{k},
\end{equation*}
then
\begin{equation*}f_{\mathcal A}(\xx,\yy)=\sum_{i=1}^{m}x_{i}\big(\mathcal A(\cdot,\xx,\yy,\xx)\big)_{i}=\sum_{l=1}^{n}y_{l}\big(\mathcal A(\xx,\yy,\xx,\cdot)\big)_{l}=\mathcal A(\xx,\yy,\xx,\yy).\end{equation*}
\begin{definition}\cite{han2009conditions,qi2009conditions}\label{defMeigenvalues}
Let $\mathcal A=(a_{ijkl})$ be a fourth-order PS-tensor. A scalar $\lambda\in\mathbb R$ is called an M-eigenvalue of $\mathcal A$ with the corresponding left M-eigenvector $\xx\in \mathbb R^{m}$ and right M-eigenvector $\yy\in\mathbb R^{n}$, if
\begin{align}\label{Meigenvalues}
\begin{aligned}
&{\mathcal A(\cdot, \yy,\xx,\yy)=\lambda \xx, \quad \mathcal A(\xx,\yy,\xx,\cdot) =\lambda \yy,}\\
&{\text{subject~to}~~\xx^\top\xx=1,~\yy^\top\yy=1.}
\end{aligned}
\end{align}
Furthermore, $\mathcal A$ is called M-positive if the smallest M-eigenvalue of $\mathcal A$ is positive, and the M-spectral radius of $\mathcal A$ is defined by $$\rho_M(\mathcal A):=
\max\big\{|\lambda|: \lambda~\text{is~an~M-eigenvalue~of~$\mathcal A$} \big\}.$$
\end{definition}

By the definition of $f_{\mathcal A}(\xx,\yy)$ and Equation (\ref{Meigenvalues}), the greatest M-eigenvalue and its corresponding M-eigenvectors of $\mathcal A$ are the optimal solution of (\ref{optimalcondition}). On the other side, the best rank-one approximation of $\mathcal A$ is defined as follows
\begin{align}\label{M-rankone}
\begin{aligned}
\min \sum_{i,k=1}^{m}\sum_{j,l=1}^{n}\big[a_{ijkl}-\lambda\cdot(x_iy_jx_ky_l)\big]^{2}.
\end{aligned}
\end{align}
It is easy to see that the partial derivatives of the best rank-one approximation problem (\ref{M-rankone}) w.r.t. the vectors $\xx$, $\yy$ are given in the first two equations in (\ref{Meigenvalues}). Hence, solving the greatest M-eigenvalue and its corresponding M-eigenvectors for $\mathcal A$ stated in Equation~\eqref{Meigenvalues}, solving the problems of bi-quadratic optimization in Equation~(\ref{optimalcondition}), and best rank-one approximation in Equation~(\ref{M-rankone}), are all equivalent.

Two notable and related methods for solving the polynomial optimization problem (\ref{optimalcondition}) are the sum of squares approach \cite{waki2006} and semidefinite programming (SDP) relaxation approach \cite{ling2010,hu2010mp,luo2010}, in which the basic idea is to relax (or approximate) the concerned problem into SDP. However, the scale of the SDP and the computational cost grow  exponentially with the dimensionality of $\mathcal A$. This  problem can be mitigated by the block improvement method (\textbf{BIM}) introduced by Wang et al. \cite{wang2009practical}. BIM is an alternating ascent method that performs a power method iteration on two symmetric indices at each alternation. The power method is a popular solution method in tensor computation and multi-linear algebra \cite{NgMichael09}. It is well-known that the power method, originated from the computing of the dominant eigenvalue of a square matrix \cite{gloub1996matrix}, has the notable advantage of less memory and computational cost. By inheriting the significant advantages of the power method, BIM becomes a practical method for solving the bi-quadratic optimization problem (\ref{optimalcondition}) and the best rank-one approximation problem (\ref{M-rankone}) via computing the greatest M-eigenvalue with its corresponding M-eigenvectors \cite{wang2009practical}.


The global convergence of BIM can be guaranteed under convexity assumption, and this restrict can be removed via a shifted technique \cite{wang2009practical,KoldaShifted}.
Recently, Wang et al. studied using the Karush-Kuhn-Tucker (KKT) point of the concerned problem and showed that BIM can achieve linear convergence rate under certain second-order sufficient conditions, in which the shift parameter plays a critical role \cite{wang2015c}. This shift parameter can be any upper bound on M-spectral radius, and the choice may significantly affect the numerical performance of BIM~\cite{wang2015c}. Wang et al. concluded that BIM has a fast convergence rate if the shift parameter is just slightly larger than the M-spectral radius, and the algorithm would converge at a slower rate if it is too large. Therefore, to have a tighter estimate of the upper bound of maximum M-eigenvalue is important, because it can serve as the shift parameter of BIM and an approximation to the M-spectral radius, which governs the maximum stated in (\ref{optimalcondition}).

How to choose a better upper bound is a challenging problem. The shift parameter provided by Wang et al. were not able to provide a very tight upper bound, due to the complex properties of high-order tensors.  
To obtain a tight upper bound of the M-spectral radius is the focus of several recent studies \cite{li2019m,che2020,he2020m}. These results are obtained by means of tensor-based entries, similar to the method of the bound given by Wang et al., but there is generally a large gap between the bounds given by this kind of approach and the eigenvalue, as shown in \cite{lsg2020t}. For example, one of the best such upper bound was obtained by Li et. al in~\cite{li2019m}, which is smaller than that used by Wang et. al in \cite{wang2009practical}. However, the gap between Li's upper bound and the true M-spectral radius can still be quite large, which can be seen from the numerical experiments in \cite{LISG}. Therefore, it remains an important task to find a better bound in practice.

Meanwhile, although Wang et al. tested the stability and efficiency of BIM by numerical comparisons with state-of-the-art methods such as block coordinate ascent method and sequential quadratic programming, their numerical results showed that any of these methods can only produce a stationary point or a KKT point of the problem, which is not necessarily the \emph{global maximizer}, as shown the Table I and II on page 1074 in \cite{wang2015c}. This is problematic for applications that seek an exact solution.  Some recent studies have mitigated this problem by improving the quality of approximation,  e.g., see \cite{ling2017,chen2021o} and related references therein.
Unfortunately, only limited attention has been paid to the exact solution for bi-quadratic optimization problem in (\ref{optimalcondition}) and best rank-one approximation problem in (\ref{M-rankone}) in very restricted cases; even less work are done for high-dimensional cases. This deficiency in literature motivated us to further explore the exact solution in certain special but useful cases.



In this paper, we present a rigorous upper bound of the M-spectral radius for nonnegative fourth-order PS-tensors via the maximum eigenvalue of the Elasticity Structure (ES) matrices, which can be regarded as an approximation to the M-spectral radius and the solution of optimization (\ref{optimalcondition}). This upper bound can be obtained with less computational cost in a set constructed from monoid semigroups of certain nonnegative fourth-order PS-tensors. Importantly, we find that using the proposed upper bound as the shift parameter in BIM enhances its convergence speed. Furthermore, as an extension of the proposed upper bound, we derive exact solutions of the M-spectral radius and its corresponding M-eigenvectors for a class of fourth-order PS-tensors, which even includes certain PS-tensors with negative entries. Finally, as an application of the proposed upper bound, we obtain a practically testable sufficient condition for nonsingular elasticity M-tensors that satisfy the strong ellipticity condition.

\subsection{Notations}
Throughout this paper, $\mathcal B=(b_{ijkl})\in\mathbb R^{m\times n\times m\times n}$ denotes a nonnegative fourth-order PS-tensor,
and $f_{\mathcal B}(\xx,\yy)$ with $\mathcal B$ in (\ref{optimalcondition})
is rewritten as
$$f_{\mathcal B}(\xx,\yy)=\sum_{i,k=1}^{m}\sum_{j,l=1}^{n}b_{ijkl}x_{i}y_{j}x_{k}y_{l}
=\sum_{j,l=1}^{n}y_{j}y_{l}\xx^\top C_{jl}\xx=\sum_{i,k=1}^{m}x_{i}x_{k}\yy^\top D_{ik}\yy,$$
where $C_{jl}$ and $D_{ik}$ are symmetric matrices in $\mathbb R^{m\times m}$ and $\mathbb R^{n\times n}$ with entries
\begin{equation}\label{cjl}
(C_{jl})_{st}=b_{sjtl}=(C_{jl})_{ts}=b_{tjsl},
\end{equation}
\begin{equation}\label{dik}
(D_{ik})_{uv}=b_{iukv}=(D_{ik})_{vu}=b_{ivku},
\end{equation}
respectively. Herein, one have $C_{jl}=C_{lj}$ and $D_{ik}=D_{ki}$. $C_{jl}$ and $D_{ik}$ are called the Bi-quadratic Structural~(\textbf{BS}) matrices of tensor $\mathcal B$ in \cite{LISG}.
We denote
\begin{equation}\label{cd}C_{l}=\sum_{j=1}^{n}C_{jl},~~ D_{i}=\sum_{k=1}^{m}D_{ik},\end{equation}
\begin{equation}\label{CD}
\overline{C}=\frac{1}{n}\sum_{l=1}^{n}C_{l}~~\text{and}~~ \overline{D}=\frac{1}{m}\sum_{i=1}^{m}D_{i}.
\end{equation}
Clearly,
$C_{l}$ and $\overline{C}$ are symmetric matrices in $\mathbb R^{m\times m}$, and $D_{i}$ and $\overline{D}$ are symmetric matrices in $\mathbb R^{n\times n}$. We call $C_{l}$ or $D_{i}$ as Elasticity Structure(\textbf{ES}) matrices of nonnegative fourth-order PS-tensor $\mathcal B$, and $\overline{C}$ or $\overline{D}$ as Mean Elasticity Structure(\textbf{MES}) matrix of $\mathcal B$. For convenience, we use $C^{\mathcal X}_{l}$ to denote the ES matrix $C_{l}$ in (\ref{cd}) corresponding to the nonnegative fourth-order PS-tensor $\mathcal X$.

In the rest of this paper, the maximum (minimum) eigenvalue of a square matrix $A$ is denoted by $\beta_{\max}(A)$ ($\beta_{\min}(A)$), and $A\geq0$ means that matrix $A$ is positive semi-definite. When not specified, $\xx^{\ast}$ and $\yy^{\ast}$ denote the M-eigenvectors corresponding to the M-spectral radius, $\xx^{\star}$ and $\yy^{\star}$ denote the eigenvectors corresponding to  $\beta_{\max}(\overline{C})$ and $\beta_{\max}(\overline{D})$, respectively. $\otimes$ denotes the Kronecker product.
Let $\mathcal O$ denotes the zero tensor where all entries are zero, and $\mathcal I=(\delta_{ijkl})\in \mathbb R^{m\times n\times m\times n}$ denotes the fourth-order identity PS-tensor defined by
$$\delta_{ijkl}=\left\{
\begin{array}{ll}
  1 ,~ &\text{if}~i=k~\text{and}~j=l,\\
  0 ,~ &\text{otherwise}.
\end{array}\right.
$$

\section{BIM-algorithom and M-spectral radius estimation}
\subsection{BIM-algorithom}
The block improvement method (BIM) is a practical method designed to compute the greatest M-eigenvalue of a fourth-order PS-tensor \cite{wang2009practical},
which is an alternating ascent method that performs a power method iteration on two symmetric indices at each alternation.

\noindent\rule{12.125cm}{0.1em}

\noindent{\bf Algorithm I: BIM \cite[Algorithm 4.1]{wang2009practical}}

\noindent\rule{12.125cm}{0.05em}

\noindent{\bf Initial step:} Input a fourth-order PS-tensor $\mathcal A=(a_{ijkl})$ and unfold it to obtain matrix $A=(A_{st})\in\mathbb R^{mn\times mn}$, where $A_{st}=a_{ijkl}$, $s=n(i-1)+j~\text{and}~t=n(k-1)+l$.

~~\noindent{\bf Substep 1:} Take $\tau=\sum_{1\leq s\leq t\leq mn} |A_{st}|$, set $\overline{\mathcal A}=\tau\mathcal I+\mathcal A$ and unfold $ \overline{\mathcal A}$ to matrix $\overline{A}$.

~~\noindent{\bf Substep 2:} Compute the eigenvector $\bf w$ of matrix $\overline{A}$ associated with the maximum
eigenvalue and fold the eigenvector into matrix $W$, where $W_{ij}={\bf w}_k$, $i=\lceil k/n\rceil$ and $j=[(k-1)\textrm{mod}~n]+1$ .

~~\noindent{\bf Substep 3:} Compute the singular vectors ${\bf u}_{1}$ and ${\bf v}_{1}$ corresponding to the largest
singular value of matrix $W$.

~~\noindent{\bf Substep 4:} Take ${\bf x}_{0} ={\bf u}_{1}$, ${\bf y}_{0}={\bf v}_{1}$, and let $k=0$.

\noindent{\bf Iterative step:} Execute the following procedure alternatively until a certain convergence
criterion is satisfied and output $\bf x^{\ast}$, $\bf y^{\ast}$:
\begin{align*}
\overline{\bf x}_{k+l}&={\overline{\mathcal A}(\cdot, {\bf y}_{k}, {\bf x}_{k}, {\bf y}_{k}),~~{\bf x}_{k}=\frac{\overline{\bf x}_{k+1}}{\|\overline{\bf x}_{k+1}\|}},\\
\overline{\bf y}_{k+l}&={\overline{\mathcal A}({\bf x}_{k+1}, {\bf y}_{k}, {\bf x}_{k+1}, \cdot),~~{\bf y}_{k+1}=\frac{\overline{\bf y}_{k+1}}{\|\overline{\bf y}_{k+1}\|}},\\
k&={k+1}.
\end{align*}
\noindent{\bf Final step:} Output the greatest M-eigenvalue of $\mathcal A$: $\lambda^{\ast}=\overline{\mathcal A}(\bf x^{\ast},\bf y^{\ast},\bf x^{\ast},\bf y^{\ast})-\tau$, and the
associated M-eigenvectors: $\bf x^{\ast}$, $\bf y^{\ast}$.

\noindent\rule{12.125cm}{0.05em}

The convergence of BIM is guaranteed by the M-positive definiteness of the input tensor $\mathcal A$. Since the M-positive definiteness of the target tensor $\mathcal A$ is unknown and is generally difficult to be identify in practice, the greatest M-eigenvalue of $\mathcal A$ is obtained indirectly by computing the greatest M-eigenvalue of the tensor $\overline{\mathcal A}=\tau\mathcal I+\mathcal A$ with M-positive definiteness, where the M-positive definiteness of $\overline{\mathcal A}$ is determined by the shift parameter $\tau$, which is an upper bound of M-spectral radius of the target tensor $\mathcal A$.  Herein, obviously, the upper bound of M-spectral radius of the target tensor $\mathcal A$ plays a core role in the convergence of BIM. In \cite{wang2015c}, Wang et al. concluded that BIM has a fast convergence rate if the shift parameter is just slightly larger than the M-spectral radius, and the algorithm would converge at a slower rate if it is too large, which
can be shown in Figure \ref{figure1}.
\begin{figure}[htbp]\label{figure1}
\centering
\includegraphics[width=2.3in]{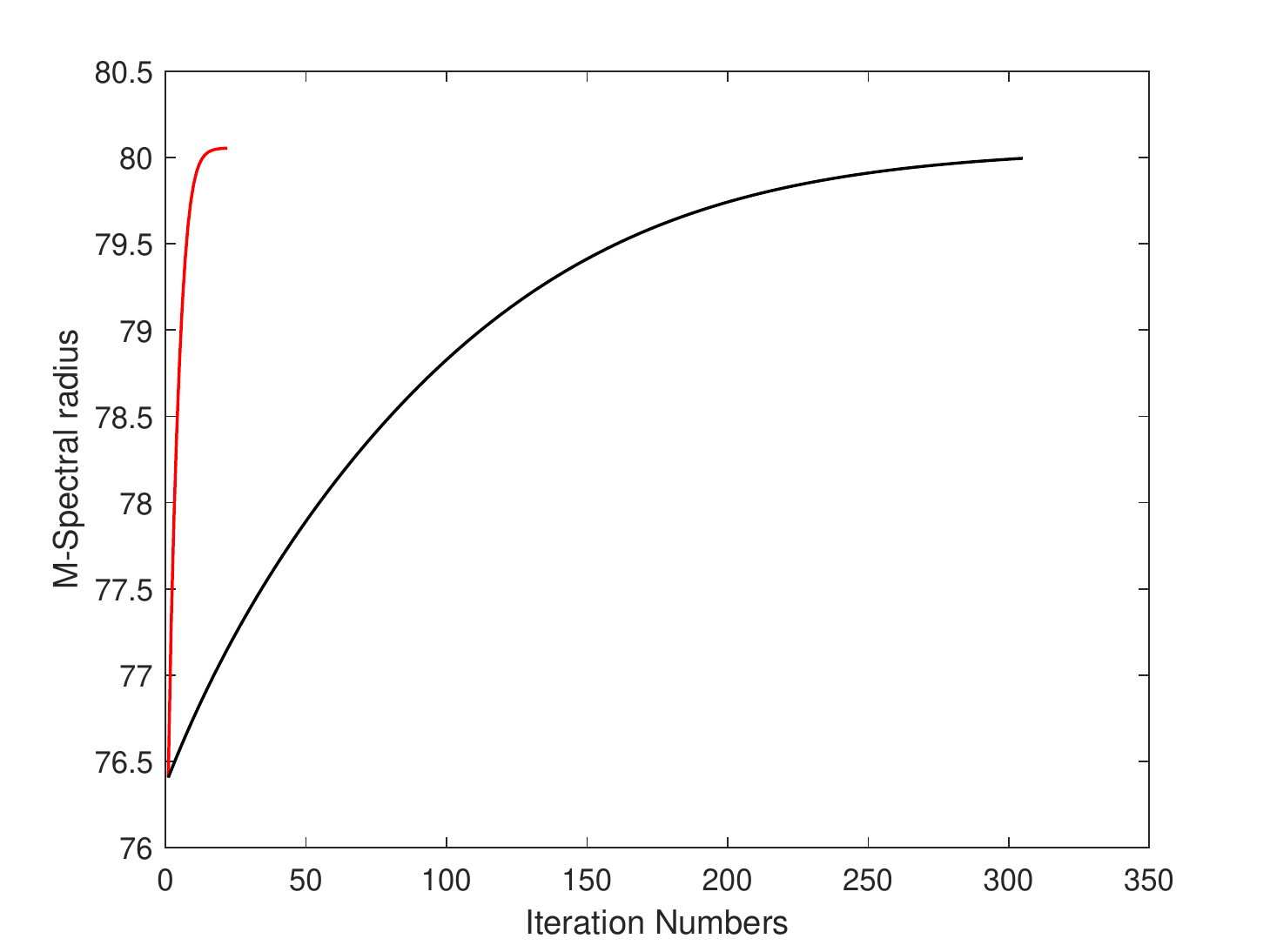}
\includegraphics[width=2.3in]{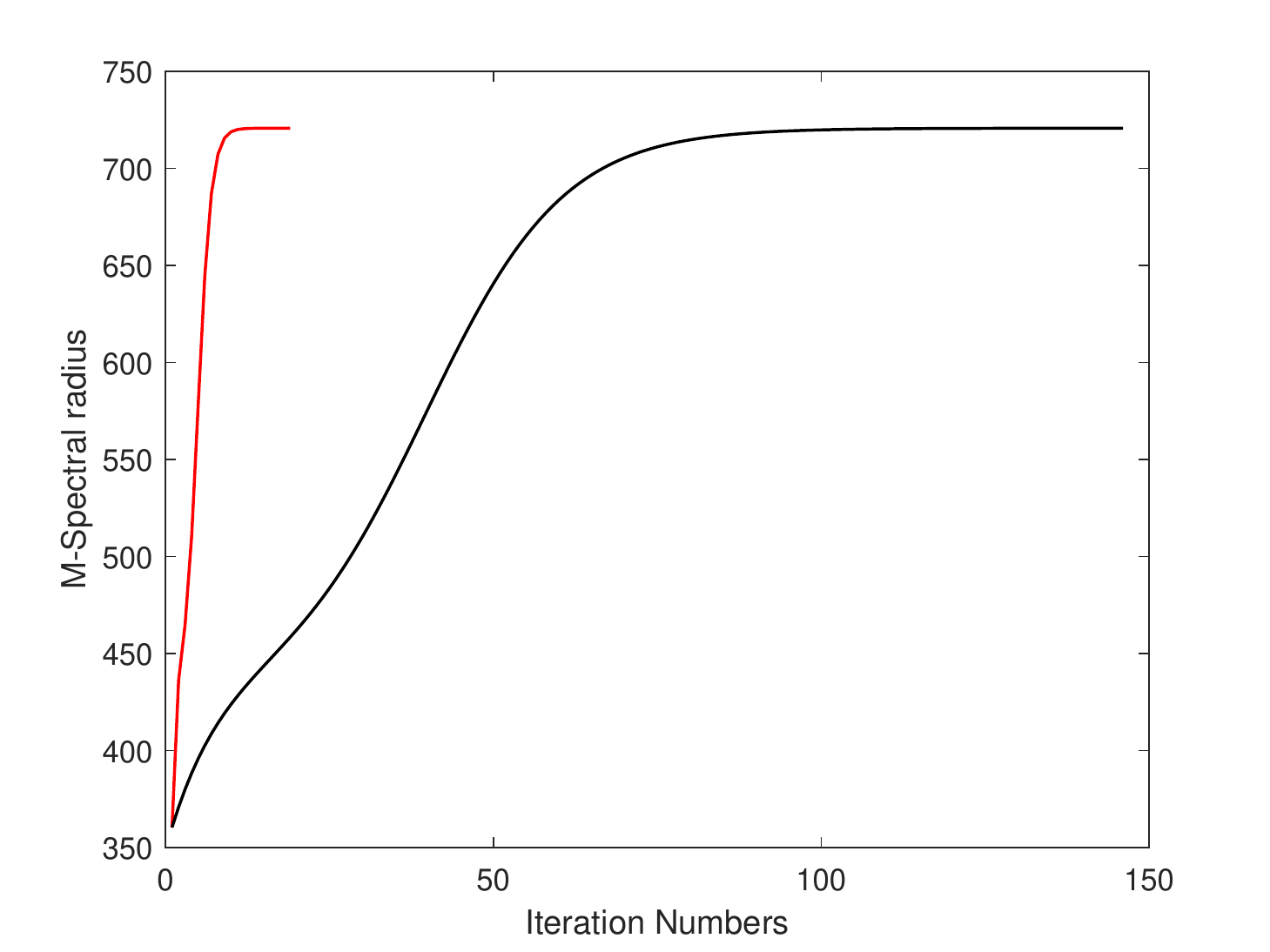}
\caption{The results of applying BIM with different shift parameters to two nonnegative fourth-order PS-tensors with the size of $m=n=10$, where the stopping condition is $1e-3$. The black curve is obtained by the shift parameter of Wang et al. in \cite{wang2009practical}, the red curve is obtained by the shift parameter we derive in Theorem \ref{Theorem-set1}.
}
\end{figure}
However, how to choose a better upper bound is an intractable problem. In the following, we derive the more rigorous bound estimate of the M-spectral radius for the nonnegative fourth-order PS-tensor.


\subsection{Rigorous bound estimate on M-spectral radius for BIM}
In this section, we first derive an upper bound of the M-spectral radius for the nonnegative fourth-order PS-tensor via the ES matrices $\big\{C_{l}\big\}_{l=1}^{n}$ or $\big\{D_{i}\big\}_{i=1}^{m}$, which is tight as shown in Figure \ref{fig:1000} at Section \ref{com_sec}. We further propose a \emph{Position Theorem} \ref{largest-distribution} such that this upper bound is more readily to computing.
Next, we derive an lower bound of the M-spectral radius for the nonnegative fourth-order PS-tensor via the MES matrix $\overline{C}$ or  $\overline{D}$. Finally,
numerical examples are given to demonstrate the efficient estimation of M-spectral radius.

First, we recall the Perron-Frobenius theorem, which is an important basis for studying nonnegative fourth-order PS-tensors.
\begin{lemma}\emph{\cite[Theorem 6]{ding2020elasticity}}\label{dingth6} Let $\mathcal B=(b_{ijkl})$ be a nonnegative fourth-order PS-tensor. Then $\rho_M(\mathcal B)$ is the greatest M-eigenvalue of $\mathcal B$, and there are a nonnegative left M-eigenvector $\xx$ and a nonnegative right M-eigenvector $\yy$ corresponding to $\rho_M(\mathcal B)$.
\end{lemma}

Based on this Lemma, we derive the following theorem via the ES matrices
$C_{l}$ and $D_{i}$.
\begin{theorem}\label{Theorem-set1}
If $\mathcal B=(b_{ijkl})$ is a nonnegative fourth-order PS-tensor, then
\begin{equation*}
\rho_{M}(\mathcal B)\leq\min\big\{R_{1}(\mathcal B),R_{2}(\mathcal B)\big\},
\end{equation*}
where $R_{1}(\mathcal B):=\max\limits_{l\in[n]}\big\{\beta_{\max}(C_{l})\big\}$, $R_{2}(\mathcal B):=\max\limits_{i\in[m]}\big\{\beta_{\max}(D_{i})\big\}$, and $C_l$, $D_i$ are defined in (\ref{cd}).
\end{theorem}
\begin{proof}
According to Lemma \ref{dingth6}, suppose that $\lambda^{\ast}$ is the greatest M-eigenvalue of $\mathcal B$,
$\xx=(x_{1},\cdots,x_{m})$ and $\yy=(y_{1},\cdots,y_{n})$ are the corresponding left and right nonnegative M-eigenvectors, respectively.
Let $y_{p}=\max\limits_{j\in [n]}y_{j}$, then $y_{p}>0$.
By the $p$-th equation of $\mathcal B(\xx,\yy,\xx,\cdot)=\lambda^* \yy$ in (\ref{Meigenvalues}), we have
\begin{equation*}
\lambda^{\ast} y_{p}=\sum_{j=1}^{n}y_{j}\big(\sum_{i,k=1}^{m} b_{ijkp}x_{i}x_{k}\big)=\sum_{j=1}^{n}y_{j}\xx^\top C_{jp}\xx.
\end{equation*}
Clearly, $y_{j}\geq 0$ and $\xx^\top C_{jl}\xx\geq 0$ for any $j,l\in[n]$.
So
$$\lambda^{\ast}=\sum_{j=1}^{n}\frac{y_{j}}{y_{p}}\xx^\top C_{jp}\xx\leq\sum_{j=1}^{n}\xx^\top C_{jp}\xx=\xx^\top \bigg(\sum_{j=1}^{n}C_{jp}\bigg)\xx=\xx^\top C_{p}\xx.$$
Since the ES matrix $C_{p}$ is symmetric, we get
$$\lambda^{\ast}\leq\beta_{\max}\big(C_{p}\big).$$
 Thus, by Lemma \ref{dingth6}, we obtain $\rho_{M}(\mathcal B)=\lambda^{\ast}\leq\beta_{\max}\big(C_{p}\big)\leq R_{1}(\mathcal B)$.

 Similarly, by using the first equation of (\ref{Meigenvalues}):
 $\mathcal B(\cdot,\yy,\xx,\yy)=\lambda^{\ast} \xx$, we can obtain $\rho_{M}(\mathcal B)=\lambda^{\ast}\leq\beta_{\max}\big(D_{q}\big)$ for some $q\in[m] $ in the same above way. So  $\rho_{M}(\mathcal B)\leq R_{2}(\mathcal B)$, which completes the proof.
\end{proof}

It can be seen from the proof of Theorem \ref{Theorem-set1} that, if we know the position $p$ (or $q$) of the largest element for the nonnegative right (or left) M-eigenvector corresponding to the greatest M-eigenvalue, then we only need to calculate the maximum eigenvalue of the ES matrix $C_p$ (or $D_q$) to obtain the upper bound for M-spectral radius. Thus, we give the following \emph{Position Theorem} to certain nonnegative fourth-order PS-tensors.

\begin{theorem}\label{largest-distribution}
Let $\mathcal B=(b_{ijkl})$ be a nonnegative fourth-order PS-tensor. If there exists $p\in [n]$ such that $\big(C_{jp}-C_{jl}\big)\geq0$ for $\forall j,l\in[n]$,
(Symmetrically, if there exists $q\in [m]$ such that $\big(D_{qk}-D_{ik}\big)\geq0$ for $\forall i,k\in[m]$,) then
$$y_{p}=\max\limits_{j\in[n]}y_{j}~~~ \textbf{(} ~x_{q}=\max\limits_{i\in[m]}x_{i}\textbf{)},$$
where $\xx$, $\yy$ are the nonnegative left and right M-eigenvectors corresponding to the greatest M-eigenvalue $\lambda^{\ast}$ of $\mathcal B$, respectively.
\end{theorem}

\begin{proof}
If $\lambda^{\ast}\neq0$, by $\mathcal B(\xx,\yy,\xx,\cdot)=\lambda^{\ast} \yy$ in (\ref{Meigenvalues}), we have
\begin{equation*}
\lambda^{\ast} y_{l}=\sum_{j=1}^{n}y_{j}\xx^\top C_{jl}\xx.
\end{equation*}
Then $\xx^\top C_{jp}\xx\geq\xx^\top C_{jl}\xx$ for $\forall j,l\in[n]$, since $\exists~p\in[n]$ such that $\big(C_{jp}-C_{jl}\big)\geq0$ (positive semi-definite)  for $\forall j,l\in[n]$.
As $\yy$ is nonnegative, we have
$$\lambda^{\ast} y_{p}=\sum_{j=1}^{n}y_{j}\xx^\top C_{jp}\xx\geq\sum_{j=1}^{n}y_{j}\xx^\top C_{jl}\xx=\lambda^{\ast} y_{l}~\text{for}~\forall l\in[n],$$
then,
$$y_{p}\geq y_{l}~\text{for}~\forall l\in[n].$$
If $\lambda^{\ast}=0$, obtained from (\ref{Meigenvalues}), $\mathcal A(\xx,\yy,\xx,\yy)=0$ for $\forall~\yy\in\mathbb R^{n}$ and $\forall~\xx\in\mathbb R^{m}$. So any unit vector $\yy\in\mathbb R^{n}$ can be a right M-eigenvector corresponding to $\lambda^{\ast}=0$.
Then there is a nonnegative unit vector $\yy$ and $\exists~p\in[n]$ such that $y_{p}\geq y_{l}~\text{for}~\forall l\in[n]$. Therefore, $y_{p}\geq y_{l}~\text{for}~\forall l\in[n]$ always holds regardless of whether $\lambda^{\ast}$ is equal to $0$ or not. Symmetrically, we have $$x_{q}\geq x_{i}~\text{for}~\forall i\in[m].$$
The proof is completed.
\end{proof}

Herein, we use $\Upsilon (\mathcal B)$ to represent the set consisting of all nonnegative fourth-order PS-tensors satisfying Theorem \ref{largest-distribution}, and we would like to know how large this set is. Obviously, the elasticity identity tensor $\mathcal I$ is a member of this set.
For the set consisting of those nonnegative fourth-order PS-tensors $\mathcal B$ in $\mathbb R^{m\times n\times m\times n}$ with respect to index $p\in[n]$ or $q\in [m]$ satisfying the conditions of Theorem \ref{largest-distribution}, 
we denote
$$\Upsilon ^{p}_{1}(\mathcal B)=\big\{\mathcal B\in \mathbb R^{m\times n\times m\times n}: \big(C_{jp}-C_{jl}\big)\geq0,~ \forall j,l\in[n]\big\},$$
$$\text{or}~\Upsilon ^{q}_{2}(\mathcal B)=\big\{\mathcal B\in \mathbb R^{m\times n\times m\times n}: \big(D_{qk}-D_{ik}\big)\geq0,~ \forall i,k\in[m]\big\}.$$
Clearly,
$$\Upsilon (\mathcal B)=\bigg(\bigcup\limits_{p=1}^{n}\Upsilon ^{p}_{1}(\mathcal B)\bigg)\bigcup \bigg(\bigcup\limits_{q=1}^{m}\Upsilon ^{q}_{2}(\mathcal B)\bigg).$$

Then, the sets $\Upsilon ^{p}_{1}(\mathcal B)$ and $\Upsilon ^{q}_{2}(\mathcal B)$ hold the following property.

\begin{theorem}\label{largest-distribution-group}
Given $p\in[n]$ and $q\in [m]$, the sets $\Upsilon ^{p}_{1}(\mathcal B)$ and $\Upsilon ^{q}_{2}(\mathcal B)$ separately form two monoid semigroups with addition (element-wise).
\end{theorem}
\begin{proof}
We only need to prove the case of set $\Upsilon ^{p}_{1}(\mathcal B)$.
First, it is easy to find that set $\Upsilon ^{p}_{1}(\mathcal B)$ contains zero tensor $\mathcal O$ as its identity element.
Next we prove the closedness of addition operation.  For
$\forall \mathcal X, \mathcal Y\in\Upsilon ^{p}_{1}(\mathcal B)$, we let $\mathcal Z=\mathcal X+\mathcal Y$. Conveniently, we use $C^{\mathcal X}_{jl}$, $C^{\mathcal Y}_{jl}$ and $C^{\mathcal Z}_{jl}$ to denote the BS matrices $C_{jl}$ corresponding to nonnegative fourth-order PS-tensors $\mathcal X$, $\mathcal Y$ and $\mathcal Z$ in (\ref{cjl}). For $\forall j,l\in[n]$, we have
\begin{equation*}
C^{\mathcal Z}_{pl}-C^{\mathcal Z}_{jl}=\big(C^{\mathcal X}_{pl}+C^{\mathcal Y}_{pl}\big)-\big(C^{\mathcal X}_{jl}+C^{\mathcal Y}_{jl}\big)=
\big(C^{\mathcal X}_{pl}-C^{\mathcal X}_{jl}\big)+\big(C^{\mathcal Y}_{pl}-C^{\mathcal Y}_{jl}\big).
\end{equation*}
Since $\mathcal X,~\mathcal Y\in\Upsilon _1(\mathcal B)$, we have
$\big(C^{\mathcal X}_{pl}-C^{\mathcal X}_{jl}\big)\geq0,~\big(C^{\mathcal Y}_{pl}-C^{\mathcal Y}_{jl}\big)\geq0,$
then, $C^{\mathcal Z}_{pl}-C^{\mathcal Z}_{jl}\geq0,$
namely,
$$\mathcal Z\in\Upsilon ^{p}_{1}(\mathcal B).$$
Finally, for $\forall \mathcal X\in\Upsilon ^{p}_{1}(\mathcal B)$, $\forall \mathcal Y\in\Upsilon ^{p}_{1}(\mathcal B)$ and $\forall \mathcal Z\in\Delta(\mathcal B)$, obviously,
$$(x_{ijkl}+y_{ijkl})+z_{ijkl}=x_{ijkl}+y_{ijkl}+z_{ijkl}=x_{ijkl}+(y_{ijkl}+z_{ijkl}),$$
that is,
$$(\mathcal X+\mathcal Y)+\mathcal Z=\mathcal X+(\mathcal Y+\mathcal Z).$$
So, set $\Upsilon ^{p}_{1}(\mathcal B)$ possess the associative law with addition operation.
The proof is completed.
\end{proof}

Note that neither the set $\Upsilon _{1}^{p}(\mathcal B)$ nor $\Upsilon _{2}^{q}(\mathcal B)$ is a group due to the nonnegative restriction. That being said, Theorem \ref{largest-distribution-group} clearly describes how large the set $\Upsilon (\mathcal B)$ is. This implies that
the \emph{Position Theorem} \ref{largest-distribution}
is applicable to a large proportion of nonnegative fourth-order PS-tensors.


Next, we derive a lower bound of the M-spectral radius for the nonnegative fourth-order PS-tensor via the MES matrix.
\begin{theorem}\label{Theorem-set2}
Let $\mathcal B=(b_{ijkl})$ be a nonnegative fourth-order PS-tensor.
Then
\begin{equation*}
\rho_{M}(\mathcal B)\geq \max\big\{\beta_{\max}(\overline{C}),\beta_{\max}(\overline{D})\big\},
\end{equation*}
where $\overline{C}$ and $\overline{D}$ are defined as the MES matrices in (\ref{CD}).
\end{theorem}
\begin{proof}
Let $\lambda^{\ast}$ be the greatest M-eigenvalue of $\mathcal B$. Since the MES matrix $\overline{C}$ is symmetric,
we obtain
\begin{align}\label{ineq-lambda}
\begin{aligned}
\lambda^{\ast}&={\max\limits_{\xx \in \mathbb R^{m}, \yy \in \mathbb R^{n}}\big\{\mathcal B (\xx,\yy,\xx,\yy) :  \xx^{\top}\xx = \yy^{\top}\yy = 1\big\}}\\
&={\max\limits_{\xx \in \mathbb R^{m}, \yy \in \mathbb R^{n}}\bigg\{\sum_{j,l=1}^{n}\sum_{i,k=1}^{m}b_{ijkl}x_{i}y_{j}x_{k}y_{l} : \xx^{\top}\xx = \yy^{\top}\yy = 1\bigg\}}\\
&\geq\max\limits_{\xx \in \mathbb R^{m}}{\bigg\{\frac{1}{n}\sum_{j,l=1}^{n}\sum_{i,k=1}^{m}b_{ijkl}x_{i}x_{k}:\yy
=(\frac{1}{\sqrt{n}},\frac{1}{\sqrt{n}},\cdots,\frac{1}{\sqrt{n}})^{\top}, \xx^{\top}\xx = 1\bigg\}}\\
&={\max\limits_{\xx \in \mathbb R^{m}}\bigg\{\frac{1}{n}\sum_{j,l=1}^{n}\xx^{\top}C_{jl}\xx: \xx^{\top}\xx = 1\bigg\}=\max\limits_{\xx \in \mathbb R^{m}}\big\{\xx^{\top}\overline{C}\xx: \xx^{\top}\xx = 1\big\}}\\
&={\beta_{\max}(\overline{C}).}
\end{aligned}
\end{align}
By Lemma \ref{dingth6}, we obtain $\rho_{M}(\mathcal B)=\lambda^{\ast}\geq\beta_{\max}(\overline{C})$. Symmetrically, 
we can obtain $\rho_{M}(\mathcal B)=\lambda^{\ast}\geq\beta_{\max}(\overline{D})$. The proof is completed.
\end{proof}

Note that when taking $\yy=(\frac{1}{\sqrt{n}},\frac{1}{\sqrt{n}},\cdots,\frac{1}{\sqrt{n}})^{\top}$ and the eigenvector $\xx^{\star}$ corresponding to  $\beta_{\max}(\overline{C})$, $\beta_{\max}(\overline{C})$ reaches a lower bound for $\rho_{M}(\mathcal B)$,
thus it is a candidate solution of approximation for maximization problem (\ref{optimalcondition}). Similarly,  \big\{$\xx=(\frac{1}{\sqrt{m}},\frac{1}{\sqrt{m}},\cdots,\frac{1}{\sqrt{m}})^{\top}$, the eigenvector $\yy^{\star}$ corresponding to  $\beta_{\max}(\overline{D})$, $\beta_{\max}(\overline{D})$\big\} is another candidate solution of approximation for (\ref{optimalcondition}). Meanwhile, we can also choose  $\xx^{\star}$ and $\yy^{\star}$ as the initial vectors for the iteration of the BIM-algorithm.

Finally, we give an example to demonstrate the above theoretical results from a numerical point of view.

\textbf{Example 2.1}
~In this example, we consider three nonnegative fourth-order PS-tensors given as follows.

(I) $\mathcal B_1=(b_{ijkl})\in\mathbb R^{3\times 3\times 3\times 3}$, which is Example $4.2$ in \cite{wang2009practical},  where
\begin{spacing}{1.5}
\setlength{\arraycolsep}{2pt}
\tiny
$$\mathcal B_{1}^{(:,:,1,1)}=\begin{bmatrix}
    1.9832 &1.0023 &4.2525\\
    2.6721 &3.2123 &2.8761\\
    4.6384 &2.9484 &4.0319
\end{bmatrix},
\mathcal B_{1}^{(:,:,2,1)}=\begin{bmatrix}
    2.6721 &3.2123 &2.8761\\
    3.0871 &0.1393 &4.4704\\
    1.7450 &3.0394 &4.6836
\end{bmatrix},
\mathcal B_{1}^{(:,:,3,1)}=\begin{bmatrix}
   4.6384 &2.9484 &4.0319\\
   1.7450 &3.0394 &4.6836\\
   0.3741 &1.6947 &2.7677
\end{bmatrix},$$
$$\mathcal B_{1}^{(:,:,1,2)}=\begin{bmatrix}
  1.0023 &4.9748 &2.3701\\
  3.2123 &1.3024 &3.2064\\
  2.9484 &4.9946 &3.8951
\end{bmatrix},
\mathcal B_{1}^{(:,:,2,2)}=\begin{bmatrix}
    3.2123 &1.3024 &3.2064\\
    0.1393 &4.9456 &2.9980\\
    3.0394 &4.3263 &0.5925
\end{bmatrix},
\mathcal B_{1}^{(:,:,3,2)}=\begin{bmatrix}
    2.9484 &4.9946 &3.8951\\
    3.0394 &4.3263 &0.5925\\
    1.6947 &4.2633 &0.1524
\end{bmatrix},$$
$$
\mathcal B_{1}^{(:,:,1,3)}=\begin{bmatrix}
    4.2525 &2.3701 &2.4709\\
    2.8761 &3.2064 &3.4492\\
    4.0319 &3.8951 &0.6581
\end{bmatrix},
\mathcal B_{1}^{(:,:,2,3)}=\begin{bmatrix}
   2.8761 &3.2064 &3.4492\\
   4.4704 &2.9980 &0.4337\\
   4.6836 &0.5925 &4.3514
\end{bmatrix},
\mathcal B_{1}^{(:,:,3,3)}=\begin{bmatrix}
   4.0319 &3.8951 &0.6581\\
   4.6836 &0.5925 &4.3514\\
   2.7677 &0.1524 &2.2336
\end{bmatrix}.$$
\end{spacing}
(II) An elastic modulus tensor $\mathcal C=(c_{ijkl})\in\mathbb R^{3\times 3\times 3\times 3}$ in equilibrium equations $$c_{ijkl}(1 + \nabla u)u_{k,lj}=0$$
that comes form Example $2.7$ in \cite{li2019checkable}, where $u_{k}(X)(k = 1, 2, 3)$ is the
displacement field ($X$ is the coordinate of a material point in the reference configuration), $u_{k,lj}$ is the $(l,j)$-th element of Jacobian matrix for the second-order partial derivatives of the $k$-th element of $u(X)$, $\mathcal C$ belongs to the rhombic system with nine elasticities \cite{gurtin1973linear},  and its nonzero entries are
$$c_{1111}=6, c_{2222}=8, c_{3333}=10, c_{1122}=1, c_{2211}= 1, c_{2233}=2, c_{3322}=2,$$
$$c_{1133}= 3, c_{3311}=3, c_{2323}=4, c_{3223}=4, c_{2332}= 4, c_{3232}= 4, c_{1212}= 5,$$
$$c_{2112}=5, c_{1221}=5, c_{2121}=5, c_{1313}=6, c_{3113}=6, c_{1331}=6, c_{3131}=6.$$
According to \cite{li2019checkable}, the equilibrium equations are of great importance in the theory of elasticity \cite{gurtin1973linear,Knowles1975}. Li et al. has transformed $\mathcal C$ into a nonnegative elasticity tensor $\mathcal B_{2}=(b_{ijkl})\in\mathbb R^{3\times 3\times 3\times 3}$ by taking
$b_{ijkl}=\frac{1}{4}(c_{ijkl} + c_{kjil} + c_{ilkj} + c_{klij}), \forall i, j, k, l \in [3],$
the nonzero entries of $\mathcal B_2$ are
$$b_{1111}=6, b_{1212}=5, b_{1313}=6, b_{2121}=5, b_{2222}=8,$$
$$b_{2323}=4, b_{3131}=6, b_{3232}=4, b_{3333}=10.$$

(III) $\mathcal B_3=(b_{ijkl})\in\mathbb R^{3\times 3\times 3\times 3}$ is a nonnegative fourth-order PS-tensor generated by us, such that each element in $\mathcal B_3$ is uniformly sampled from integers $1,2,3,4,$ and $5$. The content of $\mathcal B_3$ is provided as follows
\begin{spacing}{1}
\footnotesize$$\mathcal B_{3}^{(:,:,1,1)}=\begin{bmatrix}
  	     1 &    1&    2\\
	     2 &    4&    5\\
	     3 &    2&    3
\end{bmatrix},~~~~~~
\mathcal B_{3}^{(:,:,2,1)}=\begin{bmatrix}
    	 2  &   4&     5\\
	     1  &   5&     1\\
	     4  &   3&     3
\end{bmatrix},~~~~~~
\mathcal B_{3}^{(:,:,3,1)}=\begin{bmatrix}
        3&     2&     3\\
	    4&     3&     3\\
	    3&     4&     1
\end{bmatrix},$$
$$\mathcal B_{3}^{(:,:,1,2)}=\begin{bmatrix}
 	     1&    3&     3\\
	     4&    1&     4\\
	     2&    3&     4
\end{bmatrix},~~~~~~
\mathcal B_{3}^{(:,:,2,2)}=\begin{bmatrix}
 	     4&     1&     4\\
	     5&     3&     2\\
	     3&     4&     2
\end{bmatrix},~~~~~~
\mathcal B_{3}^{(:,:,3,2)}=\begin{bmatrix}
  	     2&     3&     4\\
	     3&     4&     2\\
	     4&     5&     3
\end{bmatrix},$$
$$\mathcal B_{3}^{(:,:,1,3)}=\begin{bmatrix}
  	     2&     3&     3\\
	     5&     4&     3\\
	     3&     4&     3
\end{bmatrix},~~~~~~
\mathcal B_{3}^{(:,:,2,3)}=\begin{bmatrix}
	     5&     4&     3\\
	     1&     2&     4\\
	     3&     2&     3
\end{bmatrix},~~~~~~
\mathcal B_{3}^{(:,:,3,3)}=\begin{bmatrix}
	     3&     4&     3\\
	     3&     2&     3\\
	     1&     3&     5
\end{bmatrix}.
$$
\end{spacing}

The bounds obtained by Theorems \ref{Theorem-set1} and \ref{Theorem-set2}
are shown in Table $1$. In fact, we obtain $\rho_{M}(\mathcal B_1)=\lambda^{\ast}\approx26.1188$
$\text{and}~\rho_{M}(\mathcal B_3)=\lambda^{\ast}\approx27.1669$
by BIM, and $\mathcal B_2$ is a nonnegative diagonal elasticity tensor, obviously, $\rho_{M}(\mathcal B_2)=\lambda^{\ast}=10$.
\begin{center}\label{tab:ulb}
\small
\begin{tabular}{p{40mm}p{60mm}p{0mm}p{0mm}}\\
\multicolumn{2}{c}{$\textbf{Table 1}$: Upper and lower bounds of the M-spectral radius in Example 2.1.}\\\hline
~Theorems \ref{Theorem-set1} and \ref{Theorem-set2}&$~~~~~~~~~~~~~~~26.1160\leq\rho_{M}(\mathcal B_1)\leq26.5099$\\
~Theorems \ref{Theorem-set1} and \ref{Theorem-set2}&$~~~~~~~~~~\,~~~~~~6.6667\leq\rho_{M}(\mathcal B_2)\leq10$\\
~Theorems \ref{Theorem-set1} and \ref{Theorem-set2}&$~~~~~~~~~~~~~~~27.1496\leq\rho_{M}(\mathcal B_3)\leq27.7779$\\
\hline
\end{tabular}
\end{center}


\subsection{
Comparison between upper bounds for M-spectral radius}\label{com_sec}
In this section, we compare this upper bound with some of the existing ones, and show that one is tight.
\begin{figure}[htbp]\label{fig:1000}
\centering
\begin{minipage}[t]{1\textwidth}
\centering
\includegraphics[width=4.8in]{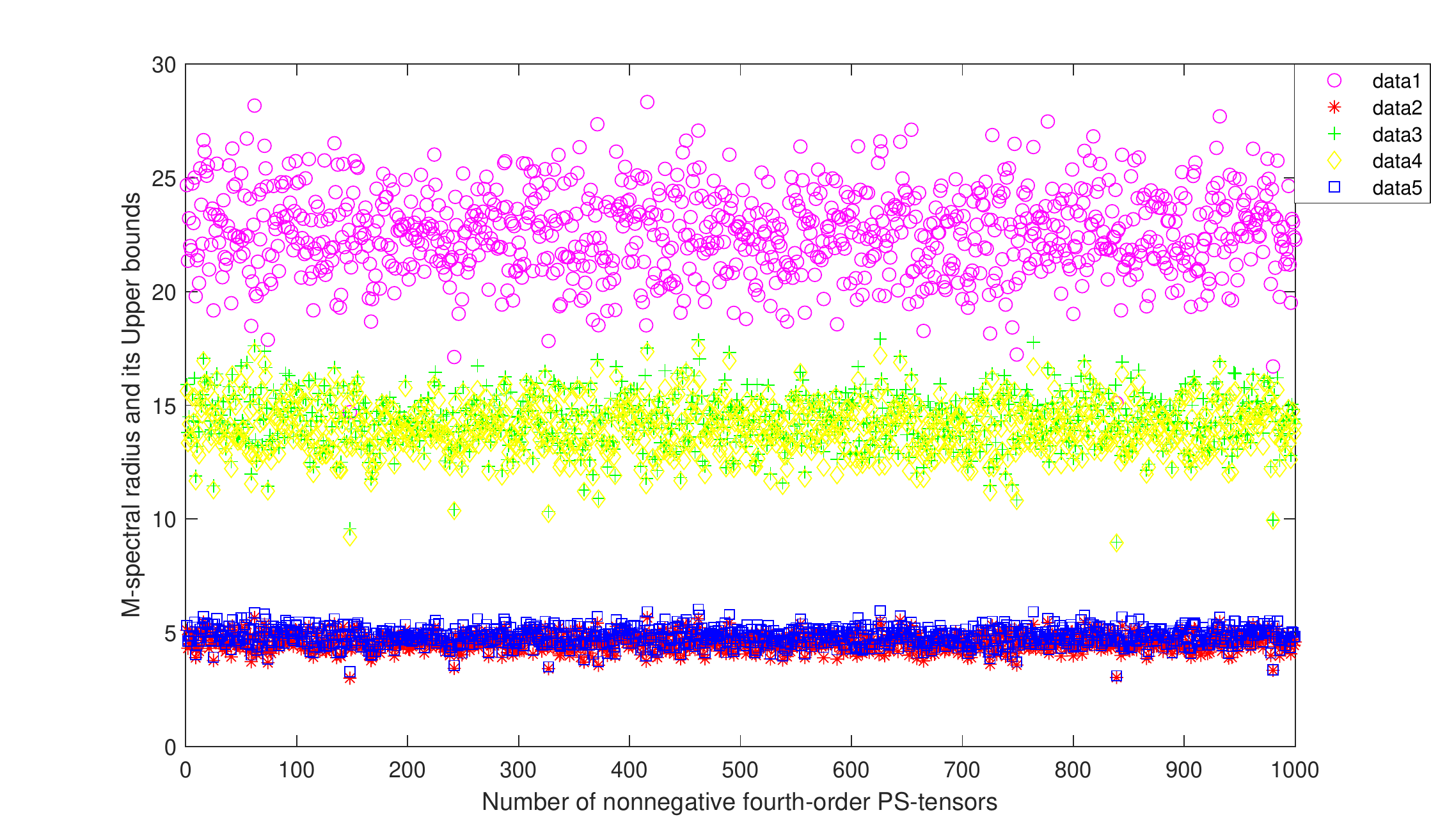}
\end{minipage}
\caption{Numerical comparison of different upper bounds for 1000 randomly generated nonnegative fourth-order PS-tensors. \textbf{data 2} represents the M-spectral radius obtained by BIM, \textbf{data 5} represents the proposed upper bounds of M-spectral radius by Theorem \ref{Theorem-set1}. \textbf{data 1}, \textbf{data 3}, and \textbf{data 4} represent the upper bounds of the M-spectral radius obtained by $\tau,~\tau_{1},~\tau_{2}$ in \cite{wang2009practical,li2019m}, separately.}
\end{figure}
First, when compared with the shift parameter $\tau$ given when Wang et al. proposed BIM, we obtain the following theorem.
\begin{theorem}\label{Theorem-comparison-1}
Let $\mathcal B=(b_{ijkl})$ be a nonnegative fourth-order PS-tensor.
Then
\begin{equation*}
\rho_{M}(\mathcal B)\leq R_1(\mathcal B)\leq\tau,~~\rho_{M}(\mathcal B)\leq R_2(\mathcal B)\leq\tau.
\end{equation*}
\end{theorem}
\begin{proof}
In fact, it is obviously to prove the following two inequalities:
\begin{equation*}
\tau=\sum_{1\leq s\leq t\leq mn} |B_{st}|\geq\sum_{1\leq i\leq k\leq m}\sum_{j=1}^{n}|b_{ijkl}|,~\text{for}~\forall l\in[n].
\end{equation*}
\begin{equation*}
\tau=\sum_{1\leq s\leq t\leq mn} |B_{st}|\geq\sum_{1\leq j\leq l\leq n}\sum_{k=1}^{m} |b_{ijkl}|,~\text{for}~\forall i\in[m].
\end{equation*}
Next, we only need to prove the case of $R_1(\mathcal B)\leq\tau$, the other case is similar.
Since the ES matrix $C_l$ is symmetric, by  Ger{\u{s}}ghorin disk theorem \cite{gloub1996matrix}, we have
$$\sum_{1\leq i\leq k\leq m}|(C_l)_{ik}|\geq\max\limits_{i\in[m]}\bigg\{\sum_{k=1}^{m}|(C_l)_{ik}|\bigg\}\geq\beta_{\max}(C_l),~~\forall l\in[n].$$
Notice that $(C_l)_{ik}=\sum\limits_{j=1}^{n}b_{ijkl}$ and $b_{ijkl}$ is nonnegative.
$$\tau\geq\sum_{1\leq i\leq k\leq m}\sum_{j=1}^{n}|b_{ijkl}|=\sum_{1\leq i\leq k\leq m}|(C_l)_{ik}|\geq\beta_{\max}(C_l),~~\forall l\in[n],$$
thus, $\tau\geq\max\limits_{l\in[n]}\big\{\beta_{\max}(C_l)\big\}=R_1(\mathcal B).$
The conclusion follows.
\end{proof}

Next, we notice that Li et al. \cite{li2019m} gave two upper bounds of the M-spectral radius in 2019, $\tau_{1}:=\min\big\{\max\limits_{i\in[m]}\Psi_{i}(\mathcal B),~\max\limits_{l\in[n]}\Gamma_{l}(\mathcal B)\big\}$ and
$\tau_{2}:=\min\big\{\Theta_{1}(\mathcal B),~\Theta_{2}(\mathcal B)\big\}$, where
{\small
\begin{equation*}
\Theta_{1}(\mathcal B):=\max\limits_{j,l\in[n]\atop j\neq l}\frac{1}{2}\left\{\sum_{i,k=1}^{m}|b_{ilkl}|+\sqrt{\big(\sum_{i,k=1}^{m}|b_{ilkl}|\big)^{2}+4\big(\Gamma_{l}(\mathcal B)-\sum_{i,k=1}^{m}|b_{ilkl}|\big)\Gamma_{j}(\mathcal B)}\right\},
\end{equation*}}
{\small
\begin{equation*}\label{p2} \Theta_{2}(\mathcal B):=\max\limits_{i,k\in[m]\atop k\neq i}\frac{1}{2}\left\{\sum_{j,l=1}^{n}|b_{ijil}|+\sqrt{\big(\sum_{j,l=1}^{n}|b_{ijil}|\big)^{2}+4\big(\Psi_{i}(\mathcal B)-\sum_{j,l=1}^{n}|b_{ijil}|\big)\Psi_{k}(\mathcal B)}\right\},
\end{equation*}}
and
\begin{equation*}
\Gamma_{l}(\mathcal B):=\sum_{j=1}^{n}\sum_{i,k=1}^{m}|b_{ijkl}|,~\forall l\in[n],\; \Psi_{i}(\mathcal B):=\sum_{k=1}^{m}\sum_{j,l=1}^{n}|b_{ijkl}|,~\forall i\in[m].
\end{equation*}

As the second comparison, we prove that the upper bound in Theorem \ref{Theorem-set1} is smaller than $\tau_{1}$ proposed by Li et al. in 2019.
\begin{theorem}\label{Theorem-comparison-2}
Let $\mathcal B=(b_{ijkl})$ be a nonnegative fourth-order PS-tensor.
Then
\begin{equation*}
\rho_{M}(\mathcal B)\leq\min\big\{R_{1}(\mathcal B),R_{2}(\mathcal B)\big\}\leq\tau_{1}.
\end{equation*}
\end{theorem}
\begin{proof}
Obviously, by  Ger{\u{s}}ghorin disk theorem \cite{gloub1996matrix}, 
$$\sum_{1\leq i\leq k\leq m}\sum_{j=1}^{n}|b_{ijkl}|=\sum_{1\leq i\leq k\leq m}|(C_l)_{ik}|\geq\beta_{\max}(C_l),~~\forall l\in[n].$$
Then
$$\Gamma_{l}(\mathcal B)=\sum_{j=1}^{n}\sum_{i,k=1}^{m}|b_{ijkl}|\geq\sum_{1\leq i\leq k\leq m}\sum_{j=1}^{n}|b_{ijkl}|\geq\beta_{\max}(C_l),~~\forall l\in[n],$$
this produced $\max\limits_{l\in[n]}\Gamma_{l}(\mathcal B)\geq R_{1}(\mathcal B)=\max\limits_{l\in[n]}\big\{\beta_{\max}(C_l)\big\}$. Symmetrically, we can obtain
$$\max\limits_{i\in[m]}\Psi_{i}(\mathcal B)\geq R_{2}(\mathcal B)=\max\limits_{l\in[m]}\big\{\beta_{\max}(D_i)\big\}.$$ Therefore,
$$\min\big\{R_{1}(\mathcal B),R_{2}(\mathcal B)\big\}\leq\tau_{1}=\min\big\{\max\limits_{i\in[m]}\Psi_{i}(\mathcal B),~\max\limits_{l\in[n]}\Gamma_{l}(\mathcal B)\big\}.$$
The proof is completed.
\end{proof}

Finally, as a remark, we conjecture that the upper bound in Theorem \ref{Theorem-set1} may be very close to the M-spectral radius of the nonnegative fourth-order PS-tensor, based on evidences gathered in our numerical experiments. Specifically, we used MATLAB software (version $2020$a) in $\mathbb R^{3\times 3\times 3\times 3}$ to randomly generate 1000 elasticity tensors uniformly sampled from $[0,1]$. In fact, it is also apparent from Figure \ref{fig:1000} that the upper bound in Theorem \ref{Theorem-set1} is smaller than $\tau$, $\tau_{1}$ and $\tau_2$ in \cite{wang2009practical,li2019m}.

\subsection{Numerical experiment with new shift parameter}\label{enhance_sec}
\begin{figure}[htbp]\label{figure2}
\centering
\includegraphics[width=2.3in]{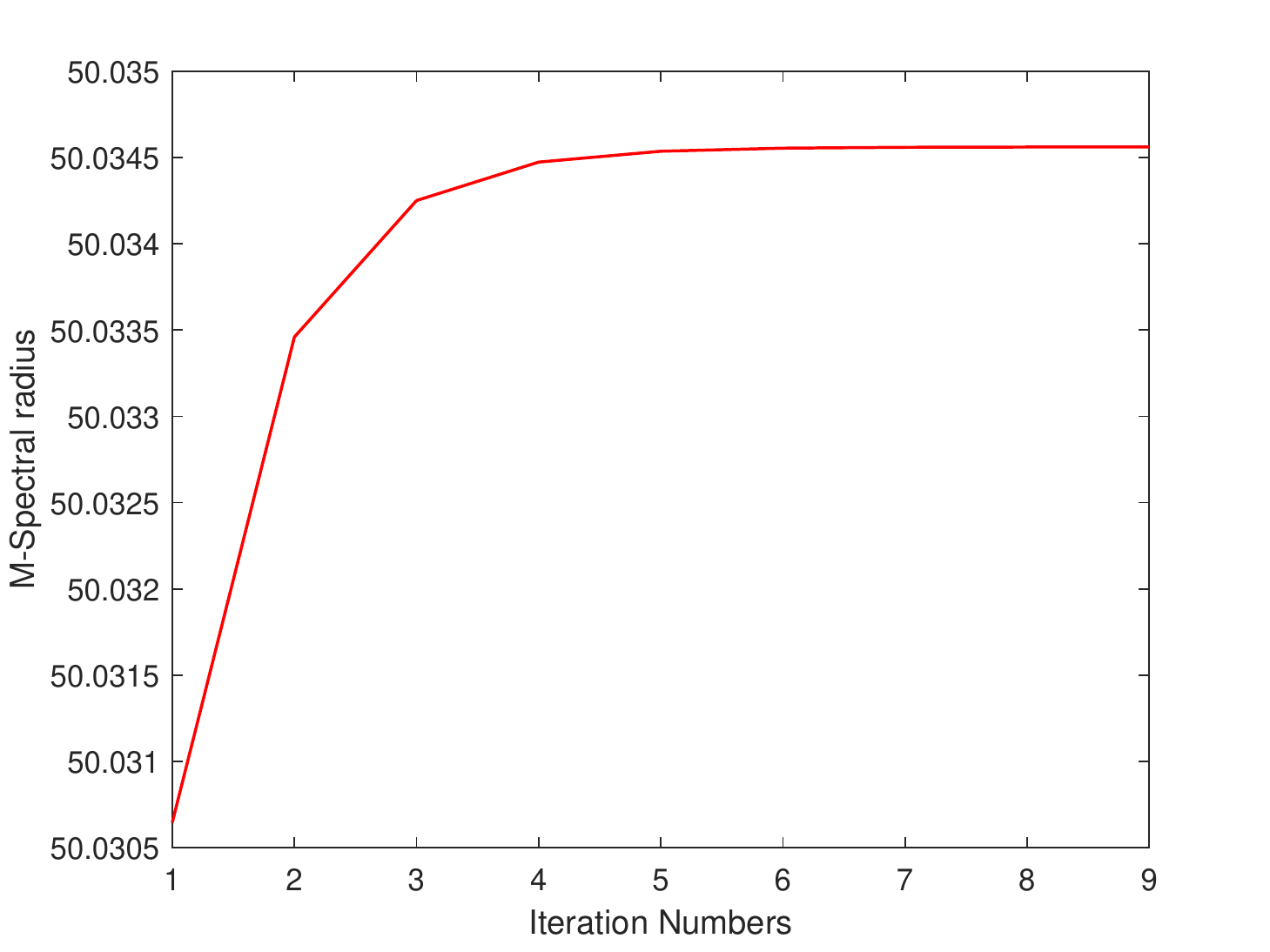}
\includegraphics[width=2.3in]{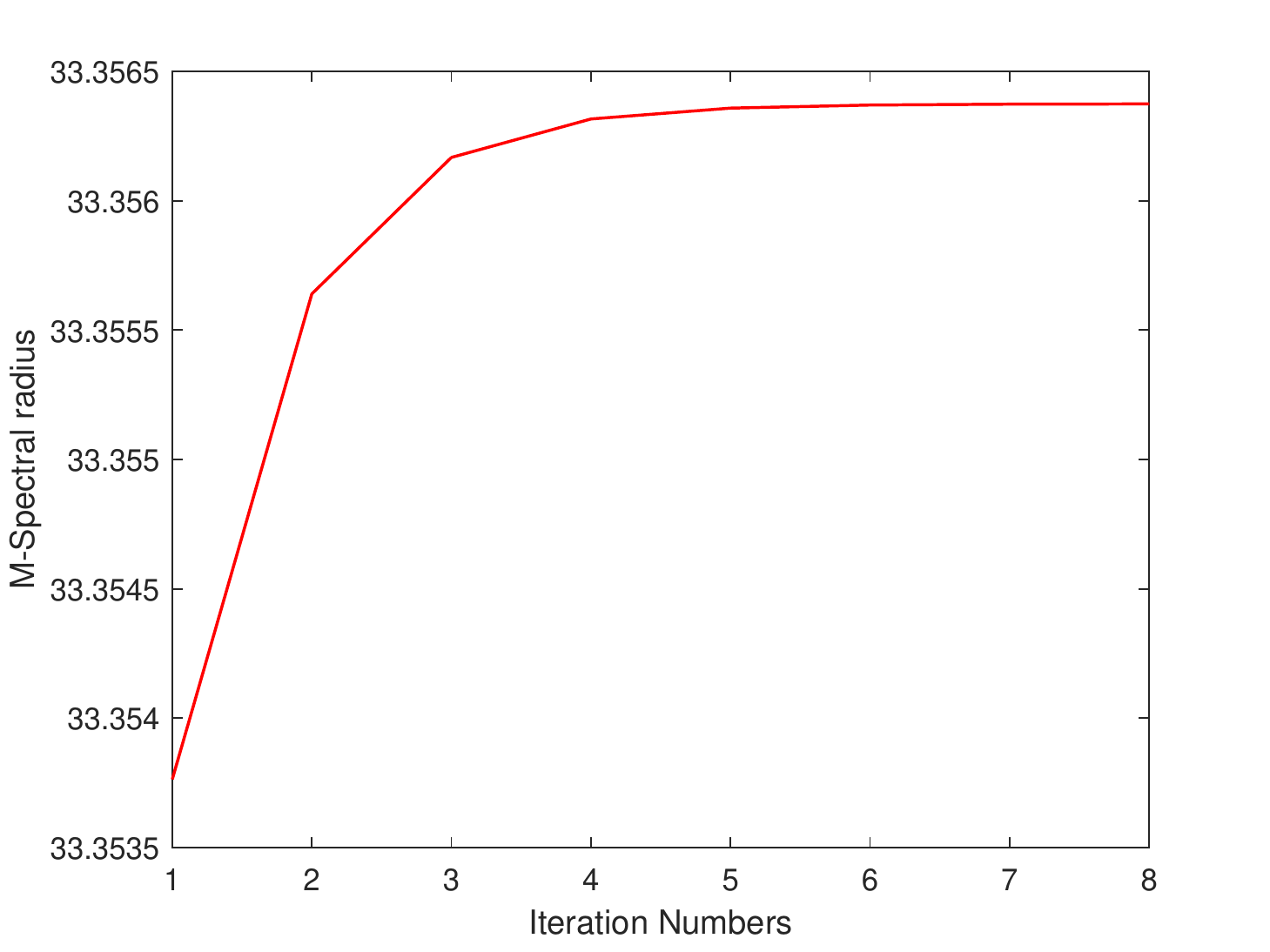}
\caption{BIM with the new shift parameter in Theorem \ref{Theorem-set1}. The stopping criterion is set to $1e-6$.
}
\end{figure}
In this section, all codes are run in
MATLAB R2020a, numerical experiment main demonstrate the significant enhancement to BIM by new shift parameter in Theorem \ref{Theorem-set1}.

First, BIM work better with new shift parameter as shown in Figure \ref{figure2}. 
In Figure \ref{figure2}, the entries of the two tensors with size $m=n=10$ are given at $h=2$ and $h=3$ by: $\mathcal A(i,p,i,q)=(1+\cos(p+q+i))/h,~\mathcal A(i,p,j,q)=(1+\sin(p+q+i+j))/h$,  where $i,j,p,q\in[10]$ and $j\neq i$. Here $h=2$ corresponds to the left figure in Figure \ref{figure2}, and $h=3$ to the right figure.
\begin{figure}[htbp]\label{fig:2}
\subfigure{
\includegraphics[scale=0.41]{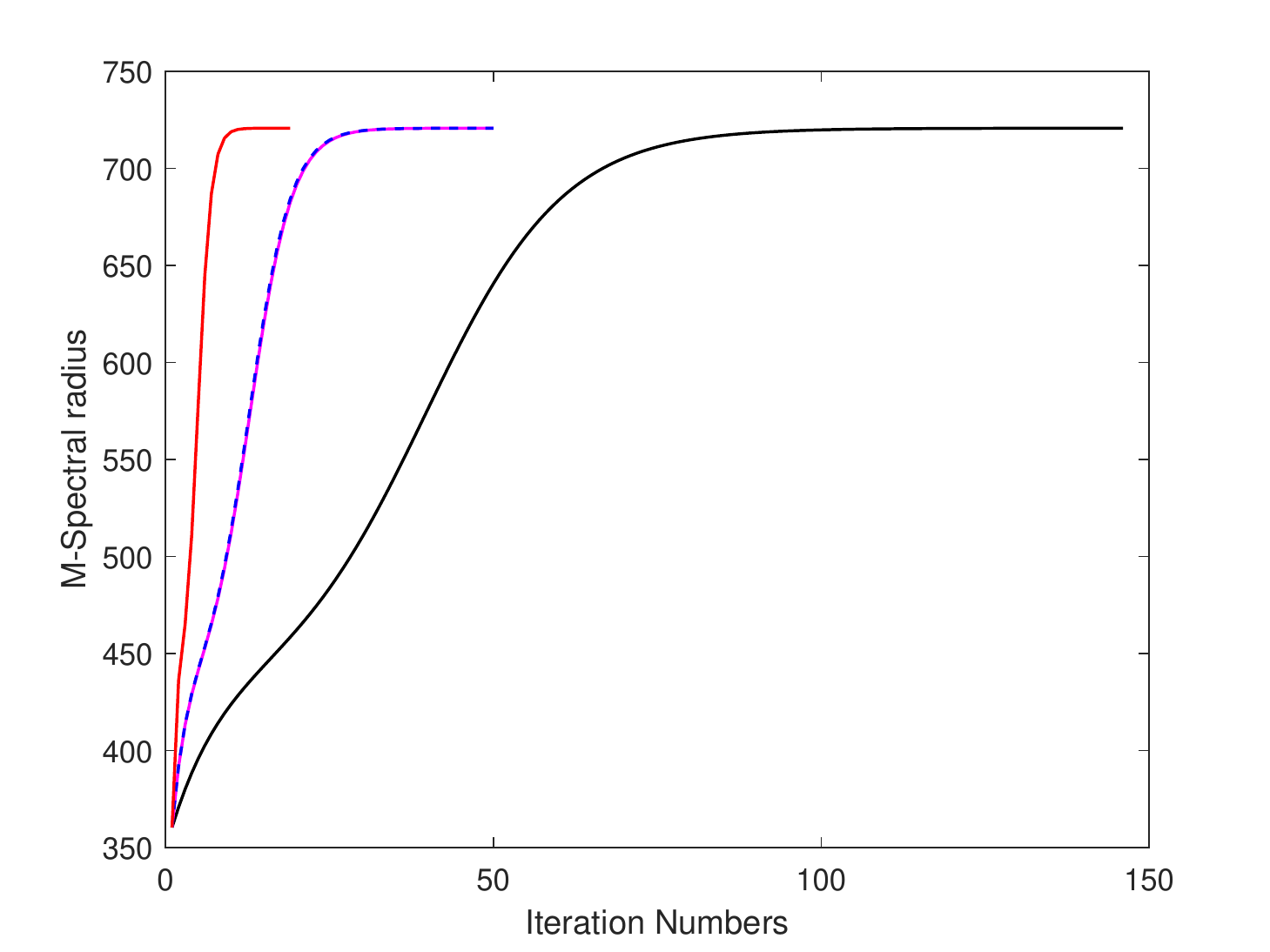} \label{3}
\includegraphics[scale=0.41]{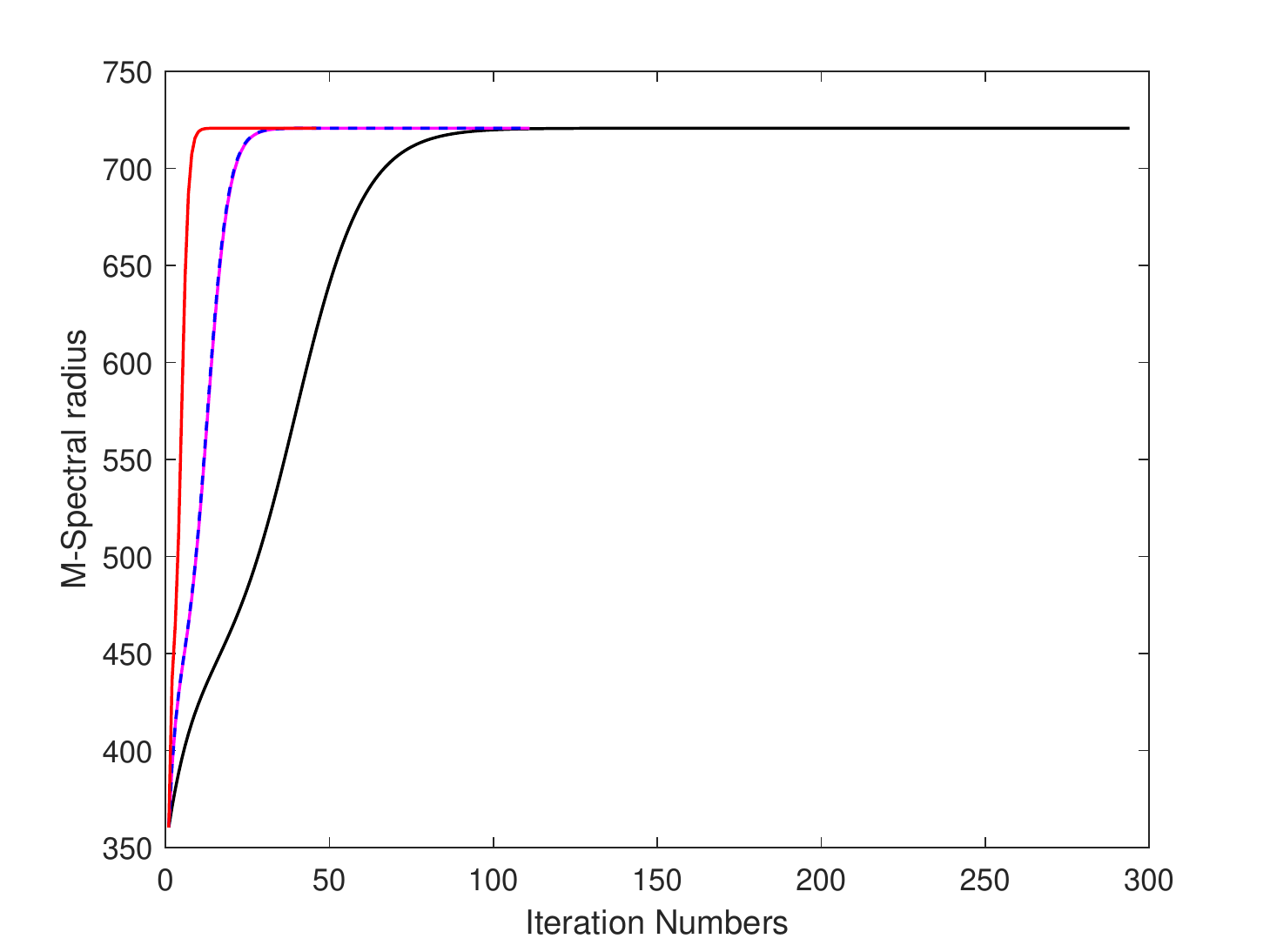} \label{4}
}
\subfigure{
\includegraphics[scale=0.41]{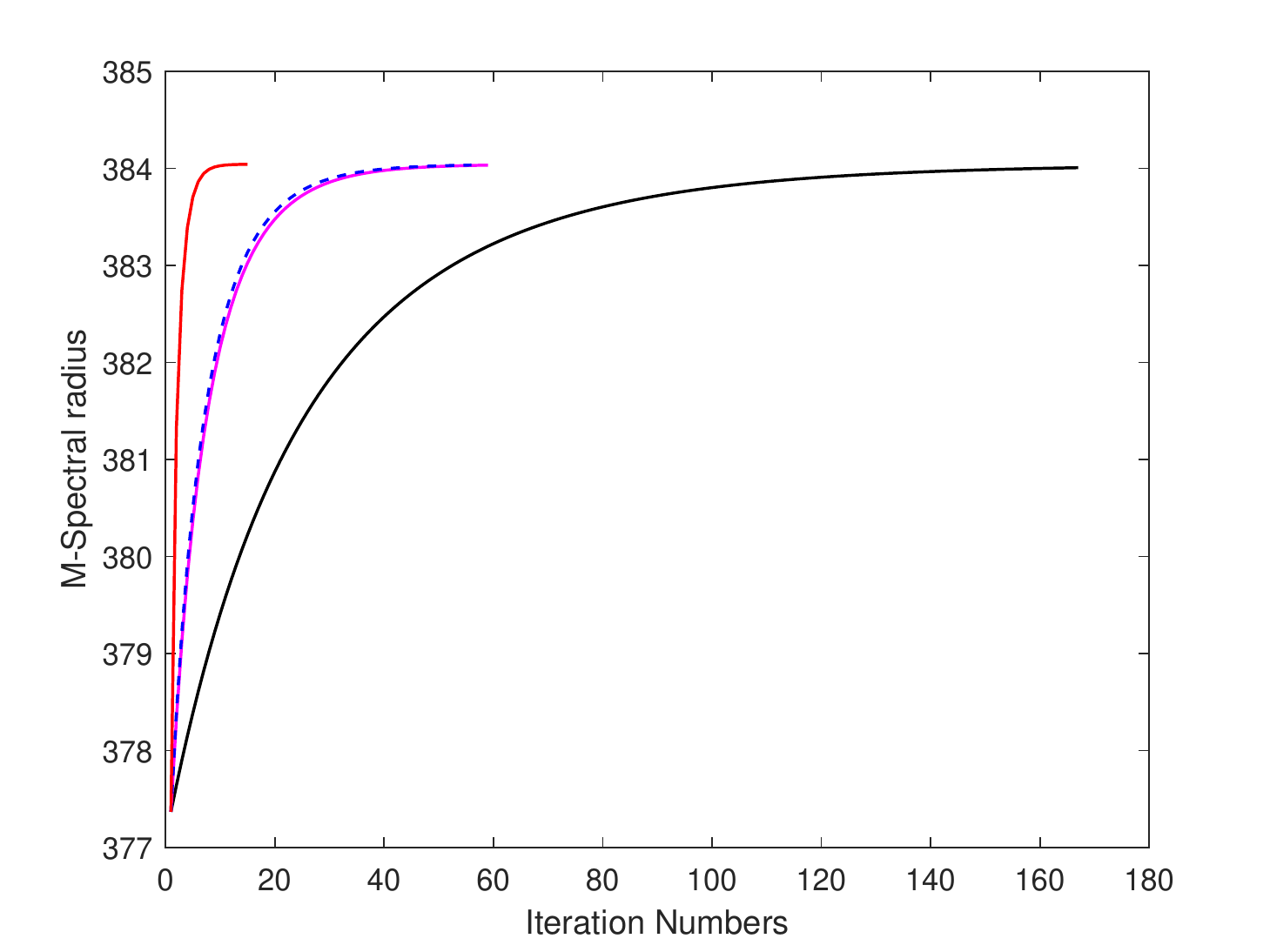} \label{7}
\includegraphics[scale=0.41]{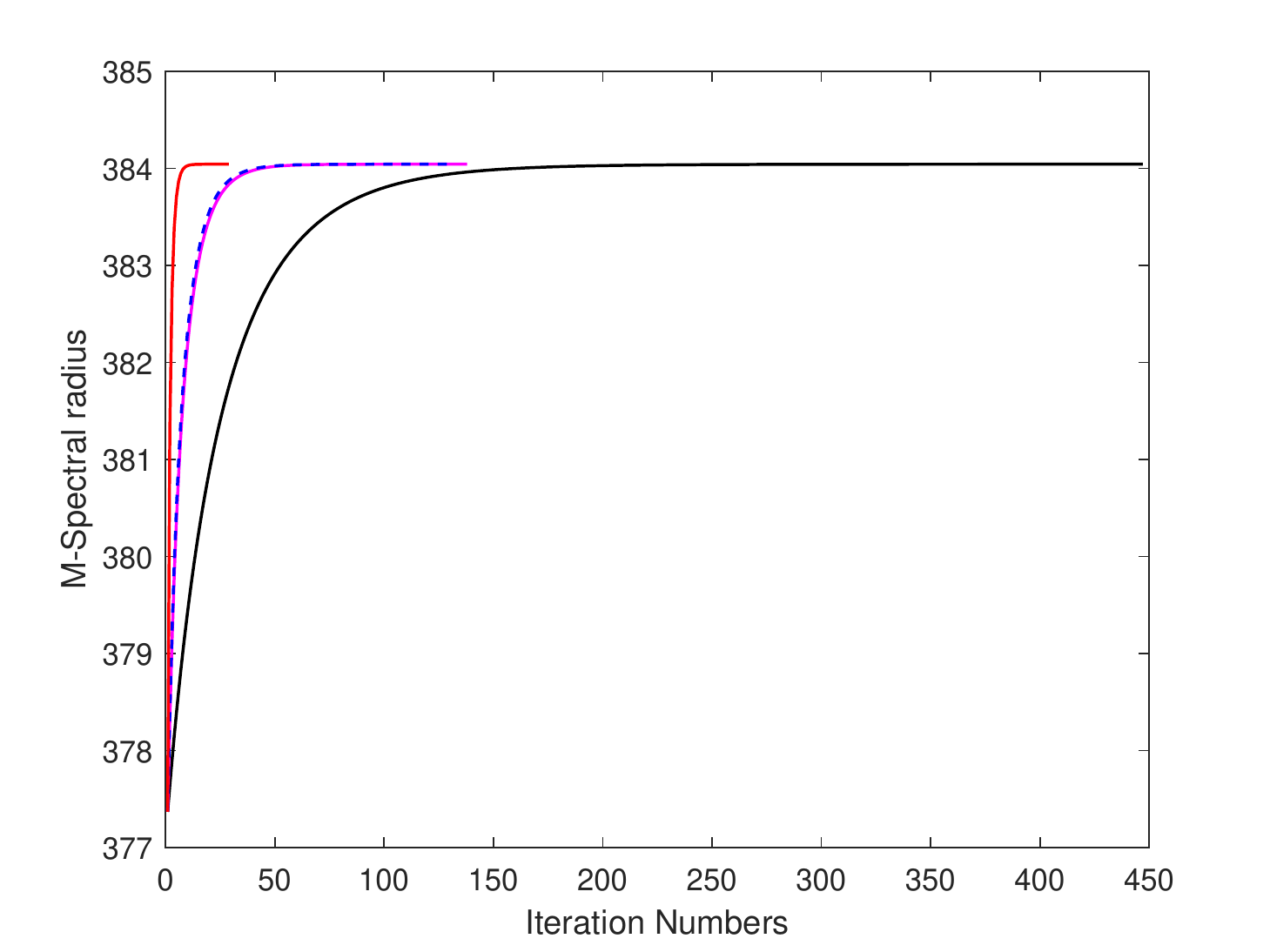} \label{8}
}
\subfigure{
\includegraphics[scale=0.41]{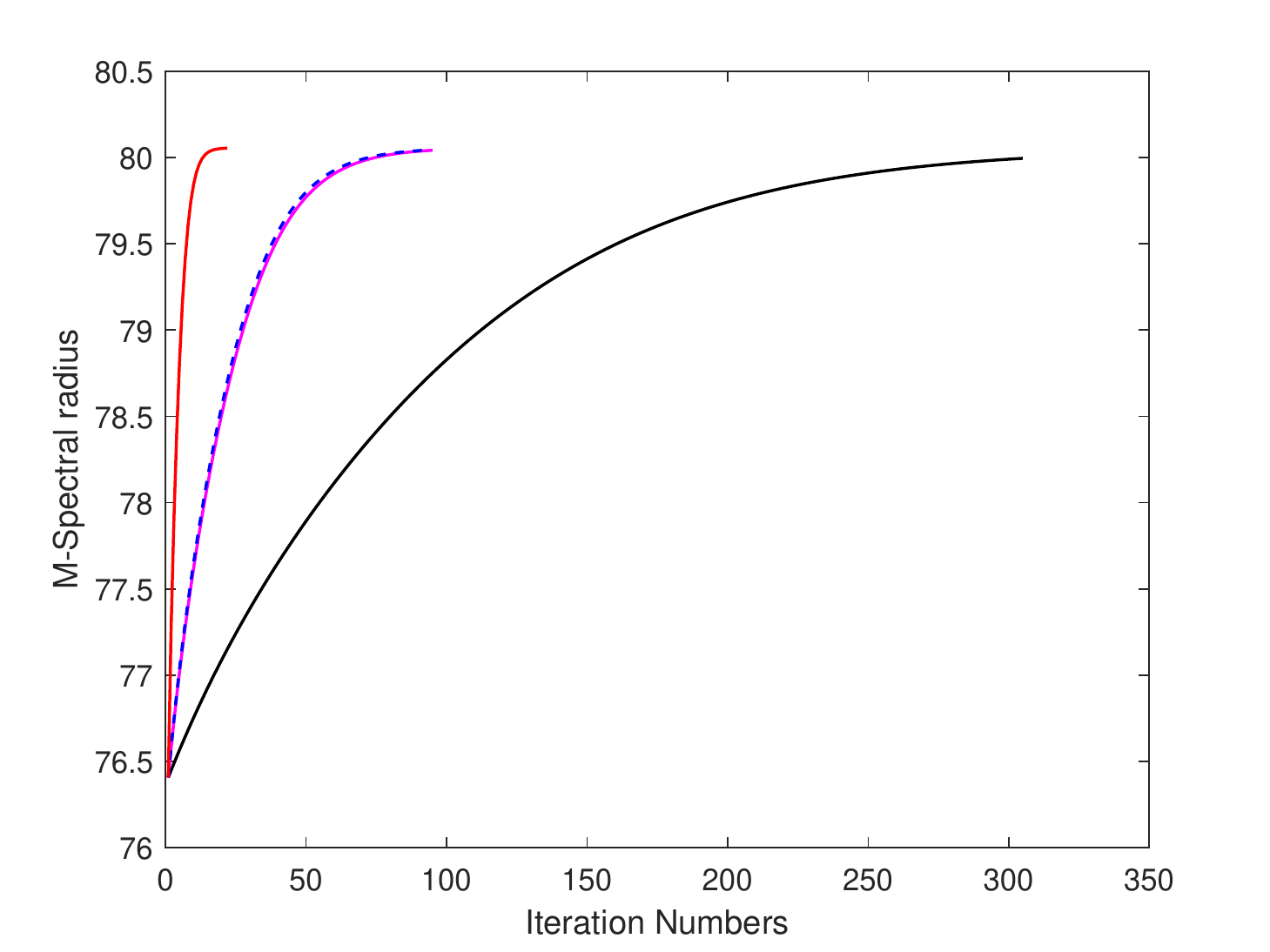} \label{1}
\includegraphics[scale=0.41]{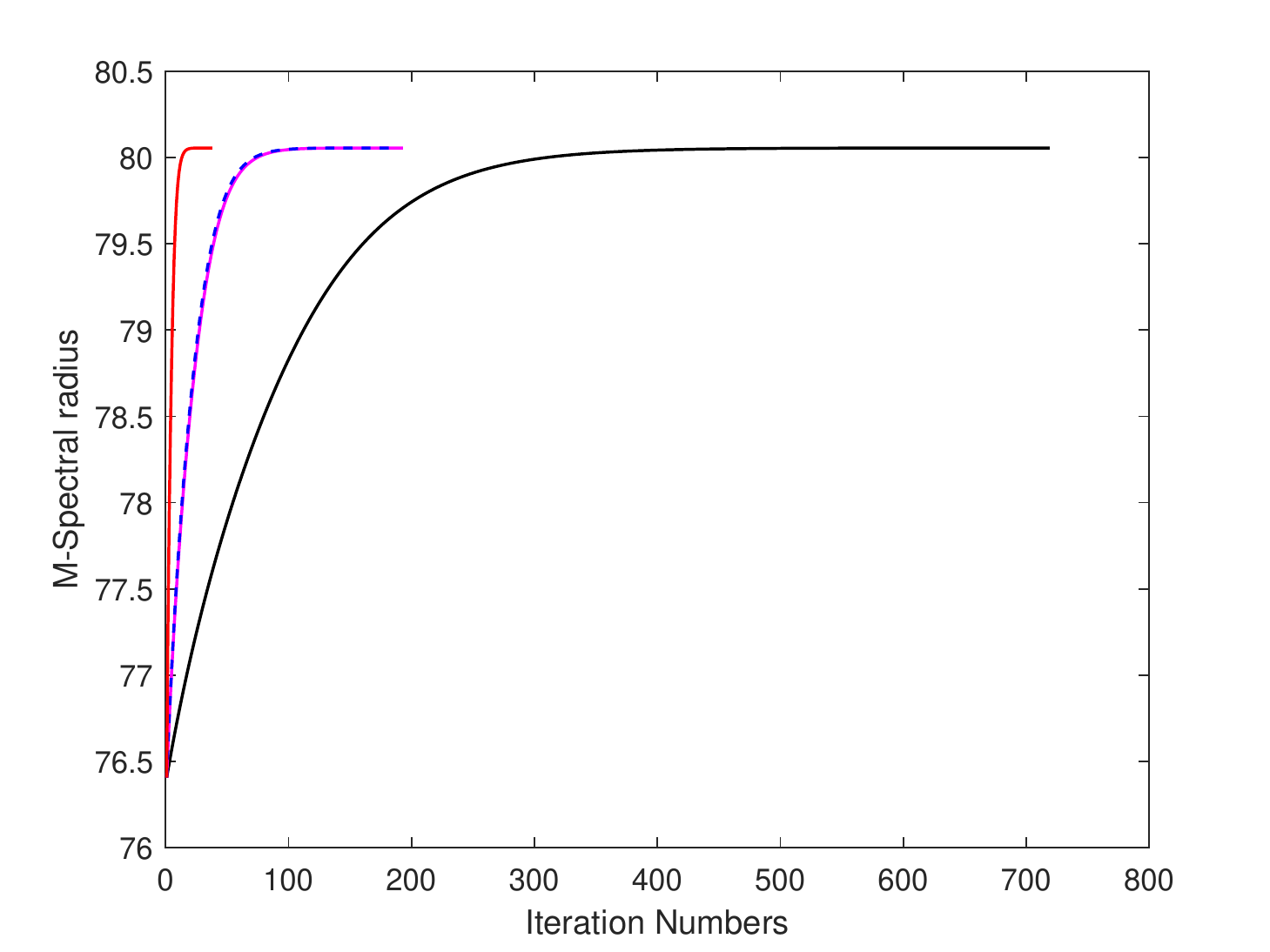} \label{2}
}
\subfigure{
\includegraphics[scale=0.41]{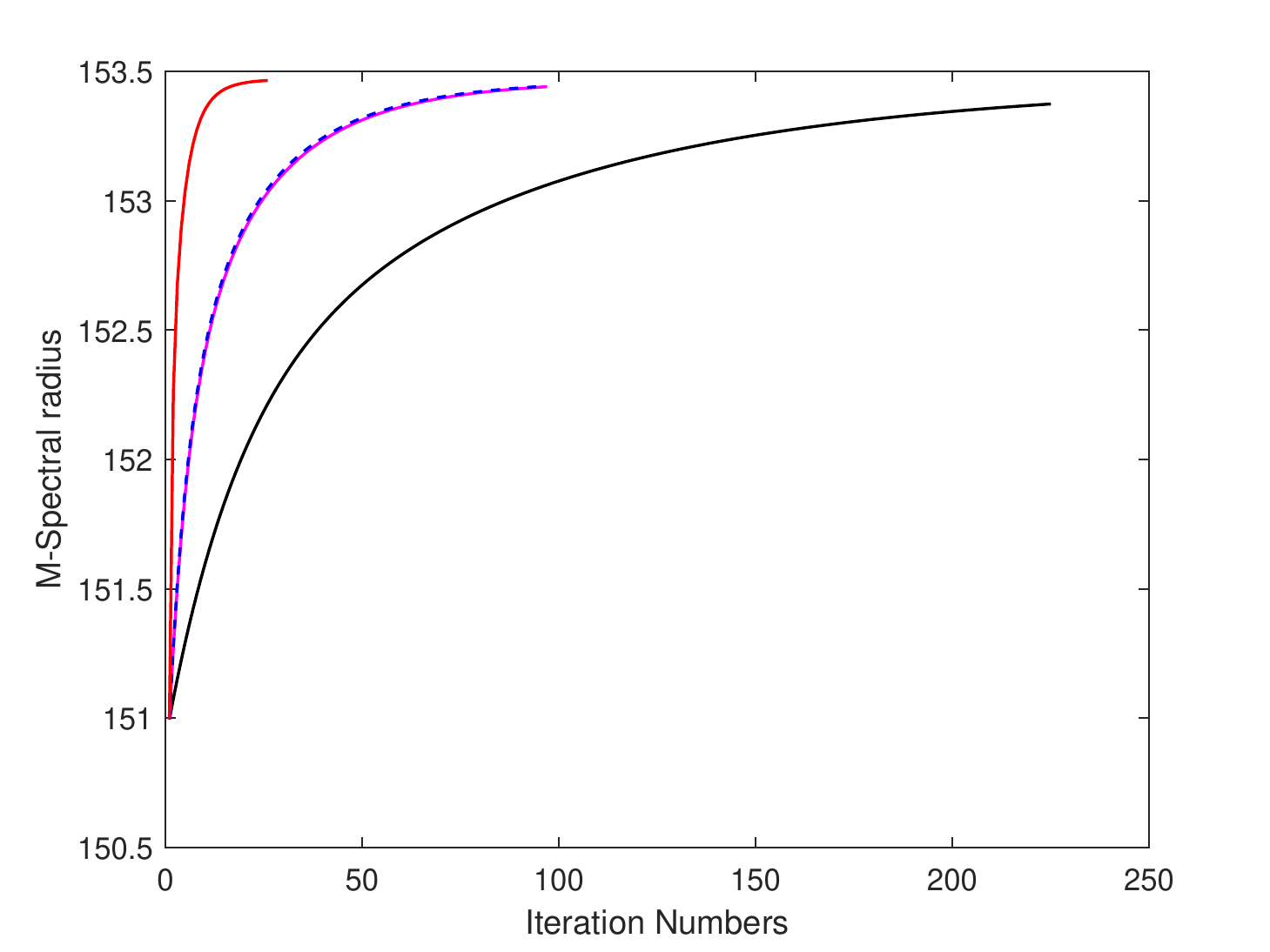} \label{5}
\includegraphics[scale=0.41]{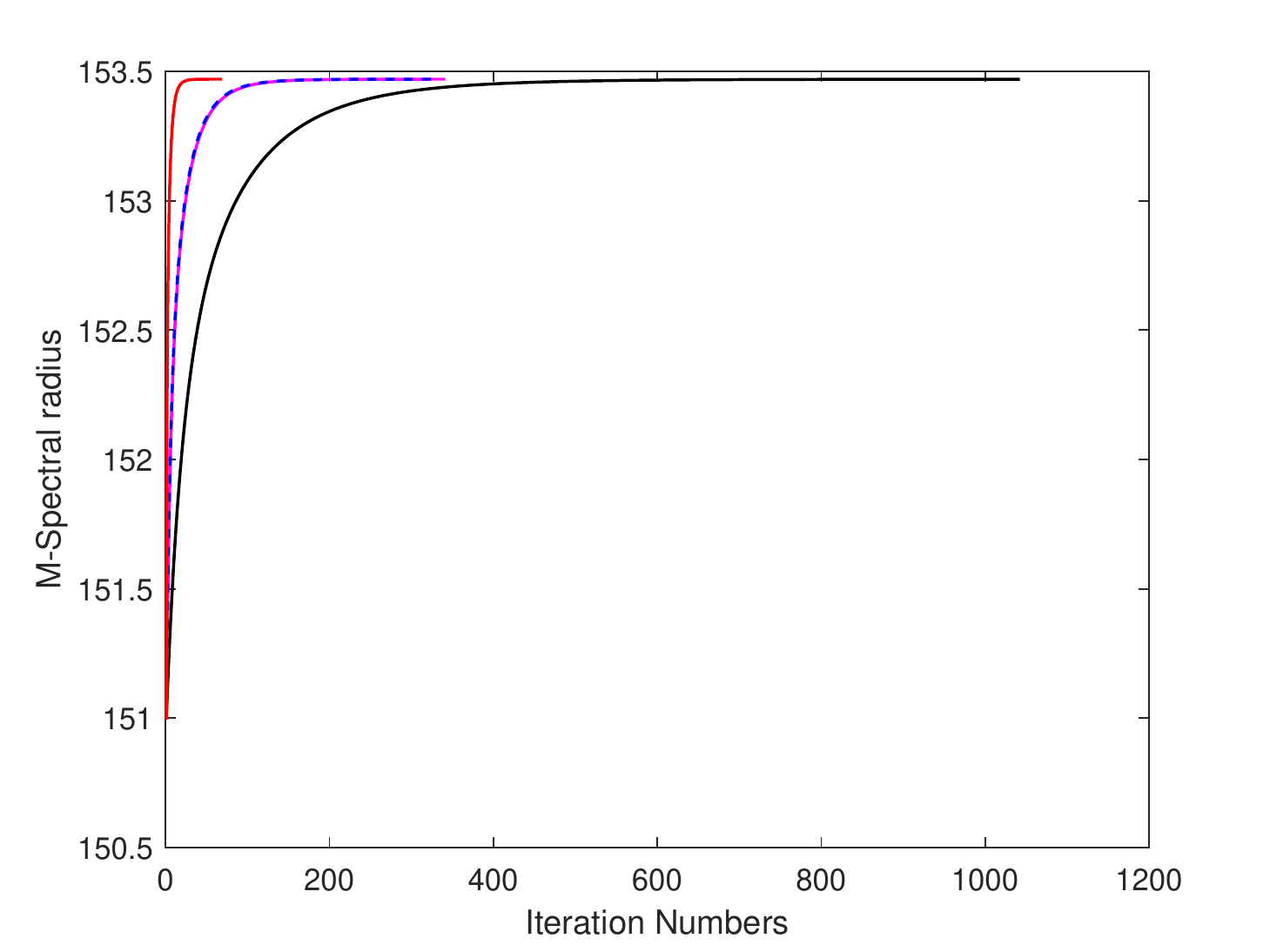} \label{6}
}
\caption{The \textbf{red curve}, \textbf{black curve}, \textbf{blue curve} and \textbf{manganese violet curve} are obtained with our upper bound in Theorem \ref{Theorem-set1} and that one $\tau$, $\tau_{1}$ and $\tau_{2}$ in \cite{wang2009practical,li2019m} as the shift parameters on BIM, separately.
The left and right stopping criteria are $1e-3$ and $1e-6$, respectively. 
}
\end{figure}
\begin{figure}[htbp]\label{figure3}
\centering
\includegraphics[width=2.3in]{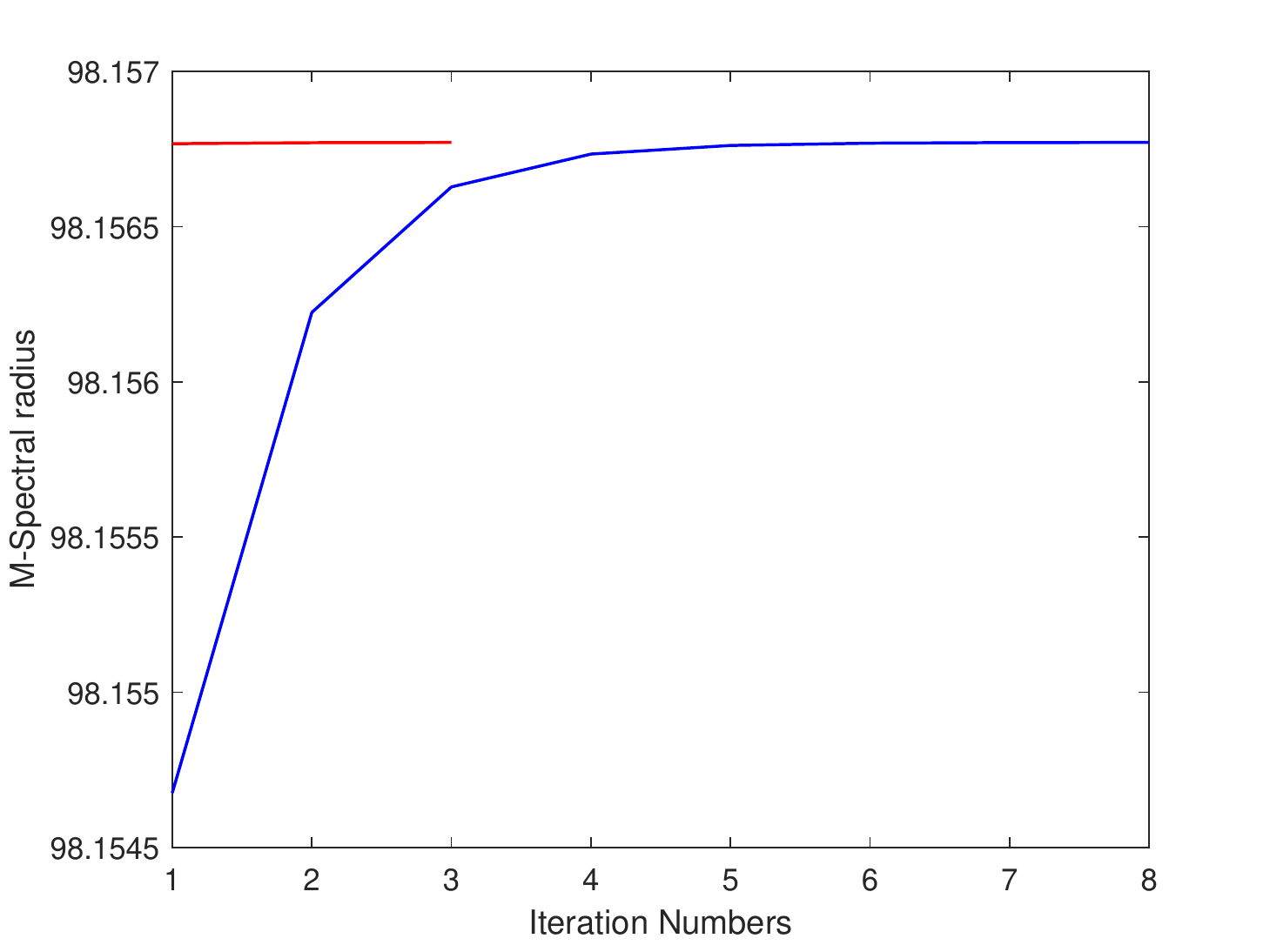}
\includegraphics[width=2.3in]{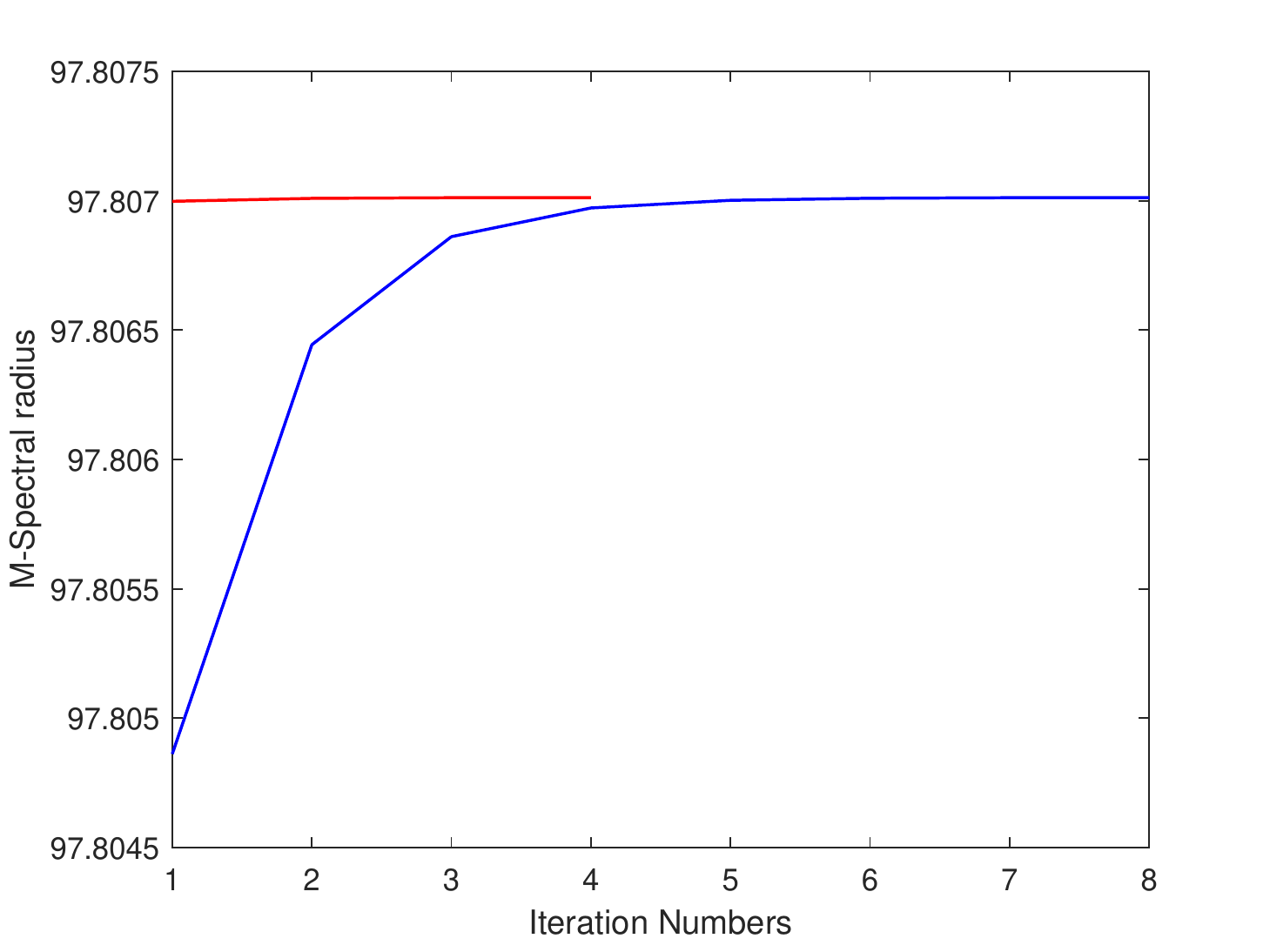}
\caption{A comparison of the different initialization vectors used for BIM under the new shift parameter. The blue curve is obtained by using the original initialization vectors of BIM, and the red curve is obtained by using $\xx^{\star}$ and $\yy^{\star}$ as the initialization vectors of BIM.
}
\end{figure}

Next, we run BIM using different upper bounds as its shift parameter under the same fourth-order PS-tensor.
The upper bound in Theorem \ref{Theorem-set1} and ones $\tau$, $\tau_{1}$ and $\tau_{2}$ of \cite{wang2009practical,li2019m} are used as the shift parameter of BIM, respectively, some numerical results are shown in Figure \ref{fig:2}. In Figure \ref{fig:2}, the entries of those tensors with size $m=n=10$ in the first three rows are given as follows: $
\mathcal A(i,p,i,q)=|\tan(1/(p+q-\sin(i)-1))|,~\mathcal A(j,p,i,q)=\mathcal A(i,p,j,q)=|\tan(1/(h(\cos(p+q)-\sin(i)-\cos(i)-\cos(j))))|$, where $h=1,3,4$;
and the entries of the tensor with size $m=n=10$ in the last row are as follows: $\mathcal A(i,p,i,q)=1/|\sin(p+q+i)|,~\mathcal A(i,p,j,q)=|\cos(p+q+i+j)|$. Here $i,j,p,q\in[10]$ and $j\neq i$. From Figure \ref{fig:2}, the convergence rate is faster under the same stopping criterion when using the upper bound in Theorem \ref{Theorem-set1} as a shift parameter in BIM than when using $\tau$, $\tau_{1}$ and $\tau_{2}$ in \cite{wang2009practical,li2019m}, see all the red curves in four subgraphs. 
Also the new shift parameter enable BIM to maintain fast convergence when the convergence accuracy is increased from $1e-3$ to $1e-6$.
This is thanks to Theorem \ref{Theorem-set1} providing a rigorous bound estimation as a shift parameter of BIM.

In fact, BIM is able to converge very quickly mainly thanks to its \textbf{Initial step}, both the shift parameter and the initialisation vectors play a critical role. Figures \ref{figure1}, \ref{figure2} and \ref{fig:2} show that we provide a better shift parameter to BIM, this is made possible by the rigorous bound estimation on M-spectral radius in Theorem \ref{Theorem-set1}. Moreover, although the initialisation vectors provided by Wang et al. \cite{wang2009practical} are already outstanding and in many cases are close to the optimal solution vector, $\xx^{\star}$ and $\yy^{\star}$ can still be used as a sub-optimal choice in many cases, especially when the dimensionality $m$ and $n$ of the tensor is relatively large. For example, for size $m=n=100$, the initialisation vectors requires the computation of the eigenvector of a matrix of size $10000\times10000$, whereas the computations of $\xx^{\star}$ and $\yy^{\star}$ are only related to the eigenvectors of two matrices of size $100\times100$. The numerical examples in Figure \ref{figure3} also shows that there are times when choosing $\xx^{\star}$ and $\yy^{\star}$ may be a good decision. In Figure \ref{figure3}, the entries of the two tensors with size $m=n=10$ are given at $h=\sin(i)$ and $h=\cos(i)$ by: $\mathcal A(i,p,i,q)=|\sin(p+q+i)+h|,~\mathcal A(i,p,j,q)=\sin(p+q+i+j)+1$,  where $i,j,p,q\in[10]$ and $j\neq i$.

\section{Exact upper bound on M-spectral radius for certain PS-tensors}
In this section, we extend the proposed upper bound on M-spectral radius to solve for the exact solution of the M-spectral radius/greatest M-eigenvalue and its corresponding M-eigenvectors of some tensors. Our motivation is that BIM and related methods, such as block coordinate ascent method and sequential quadratic programming method, can only produce a stationary point or a KKT point of the problem, which is not necessarily the \emph{global maximizer}. This problem was reported by Wang et al. in \cite{wang2015c}; technical details can be found in Table I and II on page 1074 of \cite{wang2015c}. It is natural to ask how we avoid this problem for fourth-order PS-tensors with comparatively large sizes. Although the study of exact M-spectral radius for general fourth-order PS-tensors is very challenging, we are able to derive exact solutions to the fourth-order PS-tensors with certain special structure.

\subsection{A class of fourth-order PS-tensors defined by the upper and lower bounds}\label{specialtensor}
In this part, we define a class of fourth-order PS-tensors whose M-spectral radius is equal to the maximum eigenvalue of the corresponding symmetric matrix. Note that the non-negative restriction can be relaxed for this class of PS-tensors.  

Before proceeding, we obtain the following corollary by Theorems \ref{Theorem-set1} and \ref{Theorem-set2} based on the analysis of the M-spectral radius.
\begin{corollary}\label{corollary-1}
Let $\mathcal B=(b_{ijkl})$ be a nonnegative fourth-order PS-tensor.
Then
\begin{align*}
 \beta_{\max}\bigg( \frac{1}{n}\sum_{l=1}^{n}C_{l}\bigg)\leq\rho_{M}(\mathcal B)\leq\max\limits_{l\in[n]}\bigg\{\beta_{\max}\big(C_{l}\big)\bigg\},\\
  ~\text{or}~
  \beta_{\max}\bigg( \frac{1}{m}\sum_{i=1}^{m}D_{i}\bigg)\leq\rho_{M}(\mathcal B)\leq\max\limits_{i\in[m]}\bigg\{\beta_{\max}\big(D_{i}\big)\bigg\},
\end{align*}
where the ES matrix $C_{l}$ or $D_{i}$ is defined in (\ref{cd}).
\end{corollary}
\begin{proof}
Review the MES matrix $\overline{C}$ or $\overline{D}$ is defined in (\ref{CD}). According to Theorems \ref{Theorem-set1} and \ref{Theorem-set2}, the conclusion follows.
\end{proof}

Interestingly enough is that we get the following theorem by Corollary \ref{corollary-1}.
\begin{theorem}\label{Theorem-set3}
Let $\mathcal B=(b_{ijkl})$ be a nonnegative fourth-order PS-tensor. Then
\begin{align*}
\rho_{M}(\mathcal B)=\beta_{\max}(\overline{C})
\end{align*}
with the corresponding right M-eigenvector $\yy^{\ast}=\big(\frac{1}{\sqrt{n}},\frac{1}{\sqrt{n}},\cdots,\frac{1}{\sqrt{n}}\big)^{\top}$
and the corresponding left M-eigenvector $\xx^{\ast}$ is equal to the eigenvector corresponding to the maximum eigenvalue of MES matrix $\overline{C}$,
if $$
\beta_{\max}(\overline{C})=\max\limits_{l\in[n]}\bigg\{\beta_{\max}\big(C_{l}\big)\bigg\}~\text{or}~
\beta_{\max}(\overline{C})=\max\limits_{i\in[m]}\bigg\{\beta_{\max}\big(D_{i}\big)\bigg\},$$
where ES matrix $C_{l}$ or $D_{i}$ is defined in (\ref{cd}), MES matrix $\overline{C}$ is defined in (\ref{CD}).
\end{theorem}
\begin{proof}
According to Corollary \ref{corollary-1} and the definition of MES matrix $\overline{C}$ in (\ref{CD}), obviously, $\rho_{M}(\mathcal B)=\beta_{\max}(\overline{C})$. 
Since $\rho_{M}(\mathcal B)=\beta_{\max}(\overline{C})$, by inequality (\ref{ineq-lambda}),
we have
$\yy^{\ast}=\big(\frac{1}{\sqrt{n}},\frac{1}{\sqrt{n}},\cdots,\frac{1}{\sqrt{n}}\big)^{\top}$ is the right M-eigenvector of $\rho_{M}(\mathcal B)$, and
the eigenvector corresponding to the maximum eigenvalue of $\overline{C}$ is the left M-eigenvector of $\rho_{M}(\mathcal B)$. The proof is completed.
\end{proof}

Similar to the proof in Theorem \ref{Theorem-set3}, we have the following theorem.
\begin{theorem}\label{Theorem-set3-2}
Let $\mathcal B=(b_{ijkl})$ be a nonnegative fourth-order PS-tensor. Then
\begin{align*}
\rho_{M}(\mathcal B)=\beta_{\max}(\overline{D})
\end{align*}
with the corresponding left M-eigenvector $\xx^{\ast}=\big(\frac{1}{\sqrt{m}},\frac{1}{\sqrt{m}},\cdots,\frac{1}{\sqrt{m}}\big)^{\top}$ and the corresponding right M-eigenvector $\yy^{\ast}$ is equal to the eigenvector corresponding to the maximum eigenvalue of MES matrix $\overline{D}$,
if
$$
\beta_{\max}(\overline{D})=\max\limits_{i\in[m]}\bigg\{\beta_{\max}\big(D_{i}\big)\bigg\}~\text{or}~   \beta_{\max}(\overline{D})=\max\limits_{l\in[n]}\bigg\{\beta_{\max}\big(C_{l}\big)\bigg\},$$
where ES matrix $C_{l}$ or $D_{i}$ is defined in (\ref{cd}), MES matrix $\overline{D}$ is defined in (\ref{CD}).
\end{theorem}

For convenience, we denote the set of nonnegative fourth-order PS-tensors satisfying Theorems \ref{Theorem-set3} and \ref{Theorem-set3-2} as $\Omega_1(\mathcal B)$ and $\Omega_2(\mathcal B)$, separately, that is
$$\Omega_1(\mathcal B):=\bigg\{\mathcal B:\beta_{\max}(\overline{C})=\max\limits_{l\in[n]}\big\{\beta_{\max}\big(C_{l}\big)\big\}~\text{or}~
\beta_{\max}(\overline{C})=\max\limits_{i\in[m]}\big\{\beta_{\max}\big(D_{i}\big)\big\}\bigg\},$$
$$\Omega_2(\mathcal B):=\bigg\{\mathcal B:\beta_{\max}(\overline{D})=\max\limits_{i\in[m]}\big\{\beta_{\max}\big(D_{i}\big)\big\}~\text{or}~   \beta_{\max}(\overline{D})=\max\limits_{l\in[n]}\big\{\beta_{\max}\big(C_{l}\big)\big\}\bigg\},$$
where ES matrix $C_{l}$ or $D_{i}$ is defined in (\ref{cd}), MES matrix $\overline{C}$ or $\overline{D}$ is defined in (\ref{CD}), and $\mathcal B$ is a nonnegative fourth-order PS-tensor.

Naturally, one key problem is raised herein: is there a tensor that satisfies Theorem \ref{Theorem-set3} or Theorem \ref{Theorem-set3-2}, and how large are $\Omega_{1}(\mathcal B)$ and $\Omega_{2}(\mathcal B)$? Obviously, it can be verified that $\mathcal I$ and $\mathcal O$  belong to $\Omega_1(\mathcal B)$ and $\Omega_2(\mathcal B)$.
In order to better answer this question, we give the following corollary to clarify Theorems \ref{Theorem-set3} and \ref{Theorem-set3-2}.
\begin{corollary}\label{corollary-2}
Let
$$\Delta_1(\mathcal B):=\big\{\mathcal B: C_{1}=C_2=\cdots=C_n \big\},$$
$$\Delta_2(\mathcal B):=\big\{\mathcal B: D_{1}=D_2=\cdots=D_m \big\},$$
where ES matrix $C_{l}$ or $D_{i}$ is defined in (\ref{cd}) and $\mathcal B$ is a nonnegative fourth-order PS-tensor. Then
$$\Delta_1(\mathcal B)\subset\Omega_1(\mathcal B),$$$$\Delta_2(\mathcal B)\subset\Omega_2(\mathcal B).$$
\end{corollary}
\begin{proof}
We only need to prove the case of $C_{1}=C_2=\cdots=C_n$.
Since $C_{1}=C_2=\cdots=C_n$, so $\overline{C}=\frac{1}{n}\sum_{l\in[n]}C_{l}=C_{1}$ and
$$\beta_{\max}(\overline{C})=\beta_{\max}(C_{1})=\max\limits_{l\in[n]}\bigg\{\beta_{\max}\big(C_{l}\big)\bigg\}.$$
The proof is completed.
\end{proof}

Obviously, it can be verified that $\Delta_1(\mathcal B)\neq\Delta_2(\mathcal B)$. It is easy to find that although the condition of set $\Delta_1(\mathcal B)$ or set $\Delta_2(\mathcal B)$ in Corollary \ref{corollary-2} is just a subset of set $\Omega_1(\mathcal B)$ or $\Omega_2(\mathcal B)$, 
but in practice, it is easy to construct a free-dimensional tensor satisfying $\Delta_1(\mathcal B)$ and $\Delta_2(\mathcal B)$ using Algorithm II.

\noindent\rule{12.125cm}{0.1em}

\noindent{\bf
Algorithm II: Randomly construct a tensor satisfying $\Delta_1(\mathcal B)$ }

\noindent\rule{12.125cm}{0.05em}

Given $m$, $n$; $\overline{C}=\text{zeros}(m,m)$; $T=\text{zeros}(n+1,m,m)$; T=tensor(T);

A nonnegative fourth-order PS-tensor $\mathcal B$ is randomly generated by $m$ and $n$;

for $i=1:n$

~~~~$B(:,i,:,i) ~\leftarrow~\text{zeros}(m,m)$;

~~~~for $j=1:n$

~~~~~~~~$T(i+1,:,:) = T(i+1,:,:) + \mathcal B(:,i,:,j)$;

~~~~end

~~~~$\overline{C}= \overline{C}+|T(i+1,:,:)-T(i,:,:)|$;~~~$\#$ $|\cdot|$, the element-wise absolute value;

end

$\mathcal B(:,i,:,i)~\leftarrow~\overline{C}- T(i+1,:,:)$; $i=1,2,\cdots,n$.

\noindent\rule{12.125cm}{0.05em}\\

Next, note again that $C_{l}=\sum_{j=1}^{n}C_{jl},~ D_{i}=\sum_{k=1}^{m}D_{ik}$ for nonnegative fourth-order PS-tensor $\mathcal B$. Thus,
$$(C_{l})_{st}=\sum_{j=1}^{n}b_{sjtl}~~\text{and}~(D_{i})_{uv}=\sum_{k=1}^{m}b_{iukv}~~\text{for}~\forall l\in[n],i\in[m].$$
Then, the subsets $\Delta_1(\mathcal B)$ and $\Delta_2(\mathcal B)$ have the following property.
\begin{theorem}\label{Theorem-set4}
The subsets $\Delta_1(\mathcal B)$ and $\Delta_2(\mathcal B)$ separately form two monoid semigroups with addition (element-wise).
\end{theorem}
\begin{proof}
It suffices to prove the case of $C_{1}=C_2=\cdots=C_n$.
Clearly, subset $\Delta_{1}(\mathcal B)$ contains zero tensor $\mathcal O$ as its identity element. Next we prove the closedness of addition operation. 
For $\forall \mathcal X=(x_{ijkl})\in\Delta_{1}(\mathcal B)$, $\forall \mathcal Y=(y_{ijkl})\in\Delta_{1}(\mathcal B)$, we let $\mathcal Z=(z_{ijkl})=\mathcal X+\mathcal Y$. For convenience, we use $C^{\mathcal X}_{l}$, $C^{\mathcal Y}_{l}$ and $C^{\mathcal Z}_{l}$ to denote nonnegative fourth-order PS-tensors $\mathcal X$, $\mathcal Y$ and $\mathcal Z$ corresponding to the ES matrix $C_{l}$ in (\ref{cd}).  For $\forall l\in[n]$, obviously,
\begin{align*}
\begin{aligned}
(C^{\mathcal Z}_{l})_{st}=\sum_{j=1}^{n}z_{sjtl}=\sum_{j=1}^{n}(x_{sjtl}+y_{sjtl})=\sum_{j=1}^{n}x_{sjtl}+\sum_{j=1}^{n}y_{sjtl}
=(C^{\mathcal X}_{l})_{st}+(C^{\mathcal Y}_{l})_{st},
\end{aligned}
\end{align*}
then $C^{\mathcal Z}_{l}=C^{\mathcal X}_{l}+C^{\mathcal Y}_{l}.$
Since $\mathcal X,~\mathcal Y\in\Delta(\mathcal B)$, thus
$C^{\mathcal Z}_{1}=C^{\mathcal Z}_{2}=\cdots=C^{\mathcal Z}_{n},$
that is,
$$\mathcal Z\in\Delta_{1}(\mathcal B).$$
Finally, for $\forall \mathcal X\in\Delta_{1}(\mathcal B)$, $\forall \mathcal Y\in\Delta_{1}(\mathcal B)$ and $\forall \mathcal Z\in\Delta_{1}(\mathcal B)$,
$$(x_{ijkl}+y_{ijkl})+z_{ijkl}=x_{ijkl}+y_{ijkl}+z_{ijkl}=x_{ijkl}+(y_{ijkl}+z_{ijkl}),$$
i.e,
$$(\mathcal X+\mathcal Y)+\mathcal Z=\mathcal X+(\mathcal Y+\mathcal Z).$$
So, $\Delta_{1}(\mathcal B)$ with the associative law for addition operation.
The proof is completed.
\end{proof}

From Corollary \ref{corollary-2} and Theorem \ref{Theorem-set4}, we know that subsets $\Delta_1(\mathcal B)$ and $\Delta_2(\mathcal B)$ always exist, so sets $\Omega_1( \mathcal B)$ and $\Omega_2( \mathcal B)$ are not only non-empty, but also in any dimension contain an infinite number of elements.

\textbf{Example 4.1}\label{3-1}~Algorithm II can generate a such tensor of arbitrary dimensionality, without losing the generality, we here consider a simple nonnegative fourth-order PS-tensor $\mathcal B_{4}=(b_{ijkl})\in\mathbb R^{2\times 2\times 2\times 2}$:
$$\mathcal B_4^{(:,:,1,1)}=\begin{bmatrix}
  	     a &~~~    b\\
	     x &~~~    y
\end{bmatrix},~
\mathcal B_4^{(:,:,2,1)}=\begin{bmatrix}
    	 x  &~~~   y\\
	     a  &~~~   b
\end{bmatrix},
$$
$$
\mathcal B_4^{(:,:,1,2)}=\begin{bmatrix}
        b&~~~     c\\
	    y&~~~     z
\end{bmatrix},~
\mathcal B_4^{(:,:,2,2)}=\begin{bmatrix}
 	     y&~~~    z\\
	     b&~~~    c
\end{bmatrix}.
$$
By (\ref{cd}) and (\ref{CD}), we have
 $$D_{12}=D_{21}=\begin{bmatrix}
    	 x  &~~~   y\\
	     y  &~~~   z
\end{bmatrix},~
 D_{11}=D_{22}=\begin{bmatrix}
  	     a &~~~    b\\
	     b &~~~    c
\end{bmatrix},~
 \overline{D}=\begin{bmatrix}
    	 a+x  &~~~   b+y\\
	     b+y  &~~~   c+z
\end{bmatrix}.
$$
Obviously, $D_{11}+D_{12}= D_{21}+D_{22}$, i.e., $\overline{D}=D_{1}= D_{2}$, then by Corollary \ref{corollary-2}, $$\mathcal B_{4}\in\Delta_2(\mathcal B)~~\text{and}~~B_{4}\in\Omega_2(\mathcal B).$$

Note that $a,~b,~c$ and $x$, $y$, $z$ can be any nonnegative real numbers, i.e, the $\mathcal B_{4}$ has six degrees of freedom.
Particularly, when $a=4$, $b=1$, $c=1$, $x=4$, $y=3$ and $z=1$, it is easy to compute, $\beta_{\max}(\overline{D})=10$, and the eigenvector corresponding to the maximum eigenvalue of $\overline{D}$ is $\big(\frac{2\sqrt{5}}{5},\frac{\sqrt{5}}{5}\big)^{\top}$. So, by Theorem \ref{Theorem-set3-2}, we have
$$\rho_{M}(\mathcal B_{4})=10,~\xx^{\ast}=\bigg(\frac{\sqrt{2}}{2},\frac{\sqrt{2}}{2}\bigg)^{\top},~\yy^{\ast}=\bigg(\frac{2\sqrt{5}}{5},\frac{\sqrt{5}}{5}\bigg)^{\top}.$$
Meanwhile, by BIM,
$$\rho_{M}(\mathcal B_{4})=\lambda^{\ast}=10,~\xx^{\ast}=\big(0.7071,~0.7071\big)^{\top},~\yy^{\ast}=\big(0.8944,~0.4472\big)^{\top},$$
which is consistent with the result of Theorem \ref{Theorem-set3-2}.


Finally, we simply extend the conclusion of Theorem \ref{Theorem-set3} to the fourth-order PS-tensor beyond non-negative classes. Also symmetrically to Theorem \ref{Theorem-set3-2}.
\begin{corollary}\label{corollary-3}
Let $\mathcal B\in\Omega_1( \mathcal B)$, $\mathcal A = \mathcal B-\eta \mathcal I$. 
Then $\lambda^{\ast}=\beta_{\max}(\overline{C})-\eta$, $\xx^{\star}$ and $\yy=(\frac{1}{\sqrt{n}},\cdots,\frac{1}{\sqrt{n}})^{\top}$ are the greatest M-eigenvalue and its corresponding M-eigenvectors of tensor $\mathcal A$.
\end{corollary}
\begin{proof}
For $\xx^\top\xx=1,\,\yy^\top \yy = 1$, obviously, $I(\xx,\yy,\xx,\yy)=1$.
Then
$$\max_{\xx,\yy\atop \xx^\top\xx=1,\,\yy^\top \yy = 1}f_{\mathcal A}(\xx,\yy)=\max_{\xx,\yy\atop \xx^\top\xx=1,\,\yy^\top \yy = 1}\mathcal A(\xx,\yy,\xx,\yy)=\max_{\xx,\yy\atop \xx^\top\xx=1,\,\yy^\top \yy = 1}\mathcal B(\xx,\yy,\xx,\yy)-\eta,$$
which leads to the conclusion.
\end{proof}

\textbf{Example 4.2}\label{3-2} Consider a fourth-order PS-tensor $\mathcal A=(a_{ijkl})=\mathcal B-4\mathcal I\in\mathbb R^{4\times 4\times 4\times 4}$, where $\mathcal B$ is a nonnegative fourth-order PS-tensor which entries given by:
$$\mathcal B(i,1,k,1)=i+k,~~\mathcal B(i,p,k,q)=\mathcal B(i,p,k,q)=i+k+p-q,$$
$$\mathcal B(i,t,k,t)=\sum_{l=1}^{4}\mathcal B(i,1,k,l)-\sum_{l=1\atop l\neq t}^{4}\mathcal B(i,t,k,l);~~i,k,p,q\in[4],~p\neq q,~t=2,3,4.$$
By (\ref{cd}) and (\ref{CD}), we have
$$
\overline{C}=\begin{bmatrix}
14&	~~18&	~~22&	~~26\\
18&	~~22&	~~26&	~~30\\
22&	~~26&	~~30&	~~34\\
26&	~~30&	~~34&	~~38
\end{bmatrix}.
$$
Thus, by Theorem \ref{Theorem-set3} and Corollary \ref{corollary-3},
we have
$$\lambda^{\ast}=\rho_{M}(\mathcal B)-4=\beta_{\max}(\overline{C})-4=106.9909-4=102.9909,$$
$$\yy^{\ast}=\bigg(\frac{1}{2},\frac{1}{2},\frac{1}{2},\frac{1}{2}\bigg)^{\top},~\xx^{\ast}=\xx^{\star}=(
     0.3825,
    0.4563,
    0.5300,
    0.6038
)^{\top},$$
where $(\lambda^{\ast}, \xx^{\ast}, \yy^{\ast})$ is the greatest M-eigenpair of tensor $\mathcal A$.

\subsection{Further discussion on the exact solution for the M-spectral radius and its corresponding M-eigenvectors}
In this part, we analyze the exact solutions of the greatest M-eigenvalue and its corresponding M-eigenvectors for another special fourth-order PS-tensors via the singular value decomposition of matrix.

As a linear transformation in $S_{m\times m}$ to $S_{n\times n}$ or $S_{n\times n}$ to $S_{m\times m}$, $\mathcal{A}$ has a natural matrix representation which we will denote by $\tilde{A} \in M_{n^{2}\times m^{2}}$, such that
\begin{equation*}
  \label{eq:A-in-matrix-form}
  \begin{gathered}
    \tilde{A}_{s\cdot} := \left(\mathrm{vec}\big( \mathcal{A}^{(:,j,:,l)} \big)\right)^{\top}, \qquad s := (j-1)n+l \in [n^{2}]. \\
    \tilde{A} \mathrm{vec}(X) = \mathrm{vec}(\mathcal{A}(X)), \qquad \mathrm{vec}(Y)^{\top} \tilde{A} \mathrm{vec}(X) = \mathcal{A}(X,Y),
  \end{gathered}
\end{equation*}
where
$$\big(\mathcal A (X)\big)_{jl}= \sum_{i,k=1}^{m} a_{ijkl} X_{ik},~~X\in S_{m\times m},$$
$${\cal A}(X,Y) =\sum_{i,k=1}^{m}\sum_{j,l=1}^{n} a_{ijkl}X_{ik}Y_{jl},~~Y\in S_{n\times n}.$$
Alternatively, we can use the following functions to compute $(j,l)$ from a given $s$
\begin{equation*}
  \label{eq:jl-from-s}
  j(s) = \left\lfloor \frac{s-1}{n} \right\rfloor +1, \qquad l(s) = s -(j(s)-1)n.
\end{equation*}

Assume that there are $R^{+}$ nonzero singular values of $\tilde{A}$, and express the singular value decomposition of $\tilde{A}$ as follows
\begin{equation*}
  \label{eq:svd-Atilde}
  \begin{gathered}
    \tilde{A} = U\Sigma V^\top, \quad U =(U_{\cdot 1},\cdots,U_{\cdot R^{+}})\in M_{n^{2}\times R^{+}}, \quad V=(V_{\cdot 1},\cdots,V_{\cdot R^{+}})\in M_{m^{2}\times R^{+}}, \\
    \qquad \Sigma = \mathrm{Diag}(\tilde{\sigma}_{1}, \dots, \tilde{\sigma}_{R^{+}}), \qquad
    U^\top U=V^\top V=I_{R^{+}}.
  \end{gathered}
\end{equation*}

By definition, we have
\begin{equation*}
  \label{eq:fA-with-SVD}
  \begin{gathered}
    \begin{split}
      f_{\mathcal{A}}(\xx,\yy) &= \mathcal{A}(\xx,\yy,\xx,\yy) = \mathcal{A}(\xx\xx^\top, \yy\yy^\top) = \mathrm{vec}(\yy\yy^\top)^\top \tilde{A} \mathrm{vec}(\xx\xx^\top) \\
      &= \mathrm{vec}(\yy\yy^\top)^\top U\Sigma V^\top \mathrm{vec}(\xx\xx^\top) = \sum_{r=1}^{R^{+}} \tilde{\sigma}_{r} \cdot \mathrm{vec}(\yy\yy^\top)^\top U_{\cdot r} \cdot V_{\cdot r}^\top\mathrm{vec}(\xx\xx^\top) \\
      &= \sum_{r=1}^{R^{+}} \tilde{\sigma}_{r} \left( \yy^\top\tilde{U}_{r}\yy \right) \left(\xx^\top\tilde{V}_{r}\xx \right),
    \end{split} \\
  \end{gathered}
\end{equation*}
Here
\begin{equation*}
  \label{eq:tildeUr-def}
  \begin{split}
    \tilde{U}_{r} &:= \mathrm{Mat}_{n\times n} (U_{\cdot r}),
    \quad \yy^\top\tilde{U}_{r}\yy = \mathrm{tr}\left( \tilde{U}_{r}\yy\yy^\top \right) = \mathrm{vec}(\yy\yy^\top)^\top U_{\cdot r}.  \\
    \tilde{V}_{r} &:= \mathrm{Mat}_{m\times m} (V_{\cdot r}),
    \quad \xx^\top\tilde{V}_{r}{\xx} = \mathrm{tr}\left( \tilde{V}_{r}\xx\xx^\top \right) = V_{\cdot r}^\top\mathrm{vec}(\xx\xx^\top).
  \end{split}
\end{equation*}
Here $\mathrm{Mat}_{n\times n}(\cdot)$ and $\mathrm{Mat}_{m\times m}(\cdot)$ are operators that turn $U_{\cdot r}$ and $V_{\cdot r}$ into square matrices, namely, they are the inverse function of the $\mathrm{vec}(\cdot)$ operator. Note that by the symmetry condition of $\mathcal{A}$, all $\tilde{U_{r}}$ and $\tilde{V_{r}}$ are real symmetric matrices, therefore they can be orthogonally decomposed in this way
\begin{equation*}
  \label{eq:UrVr-svd}
  \tilde{U}_{r} = P_{r}\mathrm{Diag}(\lambda_{r1}, \dots, \lambda_{rn}) P_{r}^\top, \qquad \tilde{V}_{r} = Q_{r} \mathrm{Diag}(\mu_{r1}, \dots, \mu_{rm}) Q_{r}^\top,
\end{equation*}
where $P_{r}$ and $Q_{r}$ are orthogonal matrices with dimensional $n$ and $m$, respectively.

Consequently, we have
\begin{equation}
  \label{eq:fA-with-double-SVD}
  \begin{split}
    f_{\mathcal{A}}(\xx,\yy) &= \sum_{r=1}^{R^{+}} \tilde{\sigma}_{r} \left( \yy^\top\tilde{U}_{r}\yy \right) \left( \xx^\top\tilde{V}_{r}\xx \right) \\
    &= \sum_{r=1}^{R^{+}} \tilde{\sigma}_{r} \left(\sum_{i=1}^{n} \lambda_{ri} (P_{r,\cdot i}^\top\yy)^{2} \right) \left(\sum_{j=1}^{m} \mu_{rj} (Q_{r,\cdot j}^\top\xx)^{2} \right) \\
    &= \sum_{r,i,j} (\tilde{\sigma}_{r}\lambda_{ri}\mu_{rj})\cdot (P_{r,\cdot i}^\top\yy)^{2} (Q_{r,\cdot j}^\top\xx)^{2},
  \end{split}
\end{equation}
where $P_{r,\cdot i}$ ($Q_{r,\cdot j}$) denotes $i$-th ($j$-th) column of orthogonal matrix $P_{r}$ ($Q_{r}$).

Although $\tilde{U}_{r}$, $r=1,\cdots, R^{+}$ are orthogonal in $M_{n\times n}$, and $P_{r,\cdot i}$ (the $i$th eigenvalue of $\tilde{U}_{r}$), for $i=1,\dots, n$ are orthogonal in $\mathbb R^{n}$, $P_{r,\cdot i}$ and $P_{r',\cdot i'}$ are in general \textbf{not independent} in $\mathbb R^{n}$ if $r\ne r'$. Therefore, in general we cannot obtain a closed-form maximum of $f_{\mathcal{A}}(\xx,\yy)$ based on Equation~\eqref{eq:fA-with-double-SVD}.

On the other hand, Equation~\eqref{eq:fA-with-double-SVD} suggests that if
\begin{enumerate}
\item $P_{r,\cdot i}$ are orthogonal  for all $r,i$ such that $\tilde{\sigma}_{r}\lambda_{ri} \ne 0$, and
\item $Q_{r,\cdot j}$ are orthogonal for all $r,j$ such that $\tilde{\sigma}_{r}\mu_{rj} \ne 0$,
\end{enumerate}
then we can maximize $f_{\mathcal{A}}({\xx},{\yy})$ among all ${\xx}$ and ${\yy}$ with unit length by selecting
\begin{equation}
  \label{eq:closed-form-max-special1}
  \begin{gathered}
    (r^{*}, i^{*}, j^{*}) := \arg\max_{r,i,j} \tilde{\sigma}_{r}\lambda_{ri}\mu_{rj}, \qquad  {\xx}^{*} = Q_{r^{*}, \cdot j^{*}}, \quad {\yy}^{*} = P_{r^{*}, \cdot i^{*}}. \\
    \rho_{M}(\mathcal{A}) := f_{\mathcal{A}}({\xx}^{*}, {\yy}^{*}) = \tilde{\sigma}_{r^{*}}\lambda_{r^{*}i^{*}}\mu_{r^{*}j^{*}}.
  \end{gathered}
\end{equation}

Below we will show how to create a nontrivial $\mathcal{A}$ with closed-form maximizer by using Kronecker product.

Let $\tilde{A} = A_{1} \otimes A_{2}$, for some $A_{1}, A_{2} \in M_{n\times m}$. Here $\otimes$ is the Kronecker product.

It is well known that the SVD of the Kronecker product can be expressed in terms of the SVD of the two individual matrices
\begin{equation*}
  \label{eq:kronecker-SVD}
  A_{1} = U_{1}\Sigma_{1}V_{1}^\top, \quad A_{2} = U_{2}\Sigma_{2}V_{2}^\top. \qquad \tilde{A} = \underbrace{\left( U_{1}\otimes U_{2} \right)}_{U} \underbrace{\left( \Sigma_{1}\otimes \Sigma_{2} \right)}_{\Sigma} \underbrace{\left( V_{1} \otimes V_{2} \right)^\top}_{V^\top}.
\end{equation*}

Therefore
\begin{equation*}
  \label{eq:Uvecyy-kronecker}
  \begin{gathered}
    U^\top \mathrm{vec} (\yy\yy^\top) = \left( U_{1}^\top\otimes U_{2}^\top \right) \mathrm{vec} (\yy\yy^\top) = \mathrm{vec}(U_{2}^\top\yy\yy^\top U_{1}).\\
    V^\top \mathrm{vec} (\xx\xx^\top) = \left( V_{1}^\top\otimes V_{2}^\top \right) \mathrm{vec} (\xx\xx^\top) = \mathrm{vec}(V_{2}^\top\xx\xx^\top V_{1}).
  \end{gathered}
\end{equation*}

\begin{equation*}
  \label{eq:fA-with-SVD-kronecker}
  \begin{split}
    f_{\mathcal{A}}(\xx,\yy) &= \mathrm{vec}(\yy\yy^\top)^\top U\Sigma V^\top \mathrm{vec}(\xx\xx^\top) = \mathrm{vec}(U_{2}^\top\yy\yy^\top U_{1})^\top \left( \Sigma_{1}\otimes \Sigma_{2} \right) \mathrm{vec}(V_{2}^\top\xx\xx^\top V_{1}) \\
    &= \mathrm{vec}(U_{2}^\top\yy\yy^\top U_{1})^\top \mathrm{vec}\left( \Sigma_{2} V_{2}^\top\xx\xx^\top V_{1} \Sigma_{1}^\top \right)\\
    &= \mathrm{vec}(U_{2}^\top\yy\yy^\top U_{1}\Sigma_{1})^\top \mathrm{vec}\left( \Sigma_{2} V_{2}^\top\xx\xx^\top V_{1} \right)
    = \uinner{T(Y)\Sigma_{1}}{\Sigma_{2}S(X)}_{F},
  \end{split}
\end{equation*}
here $T(Y) := U_{2}^\top\yy\yy^\top U_{1}$ and $S(X) := V_{2}^\top\xx\xx^\top V_{1}$, where $\uinner{\cdot~}{\cdot}_{F}$ denotes the Frobenius inner product.

By construction, $U_{1}$, $U_{2}$ are semi-orthogonal matrices, and $V_{1}$, $V_{2}$ are orthogonal matrices; $\xx$ and $\yy$ are vectors with unit length. Therefore
\begin{equation}
  \label{eq:fA-xtilde-inequality-kronecker}
  \begin{gathered}
    \|U_{1}^\top\yy\| \leq 1, \quad \|U_{2}^\top\yy\| \leq 1. \quad \|V_{1}^\top\xx\| = \|V_{2}^\top\xx\| = 1. \\
    \|T(Y)\|_{F} = \|U_{2}^\top\yy\|\cdot \|U_{1}^\top\yy\| \leq 1, \quad \|S(X)\|_{F} = \|V_{1}^\top\xx\| \cdot \|V_{2}^\top\xx\| = 1. \\
    \|T(Y)\Sigma_{1}\|_{F} \leq \tilde{\sigma}_{11}\|T(Y)\|_{F} \leq \tilde{\sigma}_{11}. \qquad
    \|\Sigma_{2} S(X)\|_{F} \leq \tilde{\sigma}_{21} \|S(X)\|_{F} = \tilde{\sigma}_{21}. \\
    f_{\mathcal{A}}(\xx,\yy) \leq \|T(Y)\Sigma_{1}\|_{F} \|\Sigma_{2} S(X)\|_{F} \leq \tilde{\sigma}_{11}\tilde{\sigma}_{21}.
  \end{gathered}
\end{equation}

As a special case, let us assume that $U_{1} = U_{2}$ and $V_{1}=V_{2}$. We can choose ${\xx}^{*} = V_{1, \cdot 1}$ and ${\yy}^{*} = U_{1, \cdot 1}$, so that $T(Y) = S(X)=E_{11}$ (a matrix in which the $(1,1)$ entry is one and all the rest entries are zero). In this case
\begin{equation}
  \label{eq:TD-DS-max}
  T(Y)\Sigma_{1} = \tilde{\sigma}_{11} E_{11}, \quad \Sigma_{2}S(X) = \tilde{\sigma}_{21} E_{11}, \quad  \uinner{T(Y)\Sigma_{1}}{\Sigma_{2}S(X)}_{F} = \tilde{\sigma}_{11}\tilde{\sigma}_{21}.
\end{equation}

By comparing Equation~\eqref{eq:TD-DS-max} with Equation~\eqref{eq:fA-xtilde-inequality-kronecker}, we conclude that $$(\xx^{*} = V_{1, \cdot 1}, \yy^{*} = U_{1, \cdot 1})$$
is a maximizer of $f_{\mathcal{A}}(\xx, \yy)$, and $\rho_{M}(\mathcal{A}) = \tilde{\sigma}_{11}\tilde{\sigma}_{21}$.

\section{Application: A criterion for nonsingular elasticity M-tensors with the strong ellipticity condition}

In this section, we derive the sufficient condition for nonsingular elasticity
M-tensors and the strong ellipticity condition of elasticity Z-tensors.

It is well know that the fourth-order PS-tensor is also  called the Elasticity Tensor \cite{li2019m,wang2018best,miao2020}.
The M-positive definiteness of an elasticity tensor is equivalent to the strong ellipticity condition in a nonlinear mechanical system governed by this elasticity tensor~\cite{zubov2011criterion}.
\begin{lemma}\emph{\cite{han2009conditions,qi2009conditions}}\label{sec}
Let $\mathcal A=(a_{ijkl})$ be an elasticity tensor, then, the strong ellipticity condition
\begin{align}\label{function}
\begin{aligned}
f_{\mathcal A}(\xx,\yy)=\mathcal A\xx\yy\xx\yy=\sum_{i,k=1}^{m}\sum_{j,l=1}^{n}a_{ijkl}x_{i}y_{j}x_{k}y_{l}>0
\end{aligned}
\end{align}
holds for all nonzero vectors $\xx\in\mathbb R^{m},~\yy\in\mathbb R^{n}$ if and only if $\mathcal A$ is M-positive definite.
\end{lemma}

Meanwhile, in the context of quantum entanglement, the role of the ``certificate'' is played by observables known as entanglement witnesses (EW). It is known (see equation $(11)$ in \cite{doherty2002distinguishing}) that $f_{\mathcal A}(\xx,\yy)\geq0$ for every product state $\mathinner{|\xx\yy\rangle}$ if the corresponding $\mathcal A$ is an EW.
Unfortunately, it is difficult to check the above necessary and sufficient condition for strong ellipticity because finding the smallest M-eigenvalue is NP-hard \cite{ling2010}. In order to overcome this difficulty, in \cite{he2020m,li2019checkable},
some sufficient conditions based directly on the elements of the elasticity tensor were obtained for the strong ellipticity condition. More relevant results could be found in
\cite{huang2018positive,qi2018tensor,LI2021} and references therein.
In \cite{ding2020elasticity}, two sufficient conditions were given for the strong ellipticity condition through an alternating projection algorithm and the unfolding matrix of elasticity tensors, and the elasticity M-tensor defined below was thoroughly investigated.

\begin{definition}\cite{ding2020elasticity}\label{defM-tensor}
An elasticity tensor $\mathcal A$ is called an \textbf{elasticity M-tensor} (a \textbf{nonsingular elasticity M-tensor}) if there exist a nonnegative elasticity tensor $\mathcal B\in \mathbb R^{m\times n\times m\times n}$
and a real number $\mu\geq\rho_{M}(\mathcal B)$ \big($\mu> \rho_{M}(\mathcal B)$\big) such that $\mathcal A=\mu\mathcal I-\mathcal B$.
Moreover, an \textbf{elasticity Z-tensor} is an elasticity tensor with all off-diagonal elements being non-positive.
\end{definition}

\begin{center}
\begin{tikzpicture}\label{fig:relation}
\filldraw[cyan!40] (8,3.4) circle (2.9);
\node [font=\fontsize{8}{6}] (node001) at (8,0.95){Elasticity tensor};
\filldraw[green!78] (6.8,3.5) ellipse (1.5 and 2.35 );
\node [font=\fontsize{8}{6}] (node001) at (6.6,4){Nonnegative};
\node [font=\fontsize{8}{6}] (node001) at (6.6,3.5){elasticity};
\node [font=\fontsize{8}{6}] (node001) at (6.6,3){tensor(NET)};
\filldraw[magenta!60](9.225,3.5) ellipse (0.9 and 2.4 );
\node [font=\fontsize{1}{2}] (node001) at (9.225,2.1){\small Elasticity};
\node [font=\fontsize{1}{2}] (node001) at (9.225,1.8){\small Z-tensor};
\filldraw[orange!60] (9.225,4.08) ellipse (0.8 and 1.8 );
\node [font=\fontsize{1}{2}] (node001) at (9.225,3){\scriptsize Elasticity};
\node [font=\fontsize{1}{2}] (node001) at (9.225,2.8){\scriptsize M-tensor};
\filldraw[blue!40] (9.225,4.5) ellipse (0.7 and 1.4 );
\node [font=\fontsize{1}{2}] (node001) at (9.225,4.85){\tiny  Nonsingular};
\node [font=\fontsize{1}{2}] (node001) at (9.225,4.65){\tiny elasticity};
\node [font=\fontsize{1}{2}] (node001) at (9.225,4.45){\tiny M-tensor};
\filldraw[green!78] (9.225,5.61) circle(0.37);
\node [font=\fontsize{1}{2}] (node001) at (9.225,5.58){\tiny  NET};
\end{tikzpicture}\\
\textsc{Fig. 6.} \emph{Inclusion relations between different elasticity tensors.}
\end{center}

Notice that it is hard to obtain $\rho_{M}(\mathcal B)$ in the above definition because the computation of the M-spectral radius is also NP-hard \cite{ling2010,wang2009practical}. Therefore, it is also difficult to verify the strong ellipticity condition by verifying that a tensor is a nonsingular elasticity M-tensor. But, based on Definition \ref{defM-tensor} and Theorem \ref{Theorem-set1}, we derive the sufficient condition for the nonsingular elasticity
M-tensor and the strong ellipticity condition of elasticity Z-tensors. 

First, by Theorem \ref{Theorem-set1}, we obtain a sufficient condition for an elasticity Z-tensor to be a nonsingular elasticity M-tensor.
\begin{theorem}\label{suff-M}
Let $\mathcal A=(a_{ijkl})$ be an elasticity Z-tensor. If there exists $\eta\ge \max\big\{a_{ijij}:i\in[m],j\in[n]\big\}$ such that
$$\eta>R_{1}(\eta\mathcal I-\mathcal A),~~or~\eta>R_{2}(\eta\mathcal I-\mathcal A),$$
then $\mathcal A$ is a nonsingular elasticity M-tensor and the strong ellipticity condition holds.
\end{theorem}
\begin{proof}
Let $\mathcal A=\eta\mathcal I-B$. Since $\mathcal A$ is an elasticity Z-tensor and $\mathcal B=\eta\mathcal I-\mathcal A$, then $\mathcal B$ is a nonnegative elasticity tensor. By Theorem \ref{Theorem-set1}, we have $R_{1}(\eta\mathcal I-\mathcal A)=R_{1}(\mathcal B)\geq\rho_{M}(\mathcal B)$. Obviously,
$$\eta>R_{1}(\eta\mathcal I-\mathcal A)\geq\rho_{M}(\mathcal B).$$
Similarly, by Theorem \ref{Theorem-set1}, we obtain
$$\eta>R_{2}(\eta\mathcal I-\mathcal A)\geq\rho_{M}(\mathcal B).$$
By Definition \ref{defM-tensor}, $\mathcal A$ is a nonsingular elasticity M-tensor.
Furthermore, an elasticity Z-tensor $\mathcal A$ is M-positive definite if and only if $\mathcal A$ is a nonsingular elasticity M-tensor \cite[Theorem 12]{ding2020elasticity}. Then, by Lemma \ref{sec}, the strong ellipticity condition holds. The proof is completed.
\end{proof}

\textbf{Example 5.1}~Consider the elasticity Z-tensor $\mathcal A=(a_{ijkl})\in\mathbb R^{2\times 2\times 2\times 2}$, whose nonzero entries are
$$a_{1111}=a_{2222}=13,~~a_{1122}=-7,~~ a_{1212}=6,~~a_{2121}=5.$$
By \cite{ding2020elasticity}, the unfolding matrix of $\mathcal A$
$$W_{x}=\left[
    \begin{array}{cccc}
    W_{x}^{(1,1)}&  W_{x}^{(1,2)}\\
   W_{x}^{(2,1)}&   W_{x}^{(2,2)}\\
    \end{array}
  \right]
=\left[
    \begin{array}{cccc}
    ~13&   ~~0&    ~~0&   -7\\
    ~~0&    ~~5&    -7&   ~~0\\
    ~~0&    -7&     ~~6&   ~~0\\
    -7&     ~~0&     ~~0&  ~13\\
    \end{array}
  \right],$$
where $W_{x}^{(j,l)}:=\mathcal A^{(:,j,:,l)}(j,l\in[2])$.
In \cite{ding2020elasticity}, Ding et al. showed that $\mathcal A$ is M-positive definite if $\beta_{\min}(W_{x})>0$. Direct computations show that $\beta_{\min}(W_{x})=\textbf{-1.5178}$, so we can not apply the positive definiteness of its unfolded matrix $W_x$ in \cite{ding2020elasticity} to judge whether the strong ellipticity condition holds for $\mathcal A$.

On the other hand, taking $\eta=\max\big\{a_{ijij}:i,j\in[2]\big\}=13$. By Theorem \ref{Theorem-set1}, $R_{1}(13\mathcal I-\mathcal A)=R_{2}(13\mathcal I-\mathcal A)=12.0623$, i.e, $13>R_{1}(13\mathcal I-\mathcal A)$ or $13>R_{2}(13\mathcal I-\mathcal A)$. So, by Theorem \ref{suff-M}, $\mathcal A$ is not only a nonsingular elasticity M-tensor, but also satisfy the strong ellipticity condition.  In fact, $\mathcal A=13\mathcal I-(13\mathcal I-\mathcal A)$, $\rho(13\mathcal I-\mathcal A)=10.7692$, i.e, $13>R_{1}(13\mathcal I-\mathcal A)=12.0623>\rho(13\mathcal I-\mathcal A)=10.7692$, and the smallest M-eigenvalue of $\mathcal A$ is about $\textbf{2.2308}$.

The above example not only demonstrates the rigor of our theoretical results, but also shows that Theorem \ref{suff-M} is useful in practice. The sufficient condition for Theorem \ref{suff-M} is relatively easy to satisfy, owing to the rigorous upper bound contributed by Theorem \ref{Theorem-set1}.

From Equation $(11)$ in \cite{doherty2002distinguishing}, in the context of entanglement, if the corresponding $\mathcal A$ is an EW, then $f_{\mathcal A}(\xx,\yy)\geq0$ for any product state $\mathinner{|\xx\yy\rangle}$.
Next, we derive two sufficient conditions for the elasticity $\mathcal A$ not to be an EW or not to be M-positive definite.
\begin{corollary}\label{M-necessi}
Let $\mathcal A=(a_{ijkl})$ be an elasticity Z-tensor. $\mathcal A$ is not an EW, if
\begin{equation*}
\sum_{i=1}^{m}\sum_{j=1}^{n}a_{ijij}<\sum_{i,k=1\atop i\neq k}^{m}\sum_{j,l=1\atop j\neq l}^{n}|a_{ijkl}|.
\end{equation*}
\end{corollary}
\begin{proof}
If the corresponding $\mathcal A$ is an EW, then $f_{\mathcal A}(\xx,\yy)\geq0$ for any product state $\mathinner{|\xx\yy\rangle}$.
Taking $\xx=\big(\frac{1}{\sqrt{m}},\frac{1}{\sqrt{m}},\cdots,\frac{1}{\sqrt{m}}\big)$ and $\yy=\big(\frac{1}{\sqrt{n}},\frac{1}{\sqrt{n}},\cdots,\frac{1}{\sqrt{n}}\big)$, then
$$f_{\mathcal A}(\xx,\yy)=\frac{1}{mn}\sum_{i,k=1}^{m}\sum_{j,l=1}^{n}a_{ijkl}\geq0,$$
thus
$$\sum_{i=1}^{m}\sum_{j=1}^{n}a_{ijij}\geq-\sum_{i,k=1\atop i\neq k}^{m}\sum_{j,l=1\atop j\neq l}^{n}a_{ijkl}.$$
Notice that $\mathcal A$ is an elasticity Z-tensor, then
$$\sum_{i=1}^{m}\sum_{j=1}^{n}a_{ijij}\geq\sum_{i,k=1\atop i\neq k}^{m}\sum_{j,l=1\atop j\neq l}^{n}|a_{ijkl}|.$$
Clearly, the conclusion follows by the view of inverse negative proposition.
\end{proof}


We generalize Corollary \ref{M-necessi} to the general elasticity tensor (fourth-order PS-tensor) $\mathcal A$ based on its proof.

\begin{corollary}
Let $\mathcal A=(a_{ijkl})$ be an elasticity tensor. $\mathcal A$ is not an EW, if
\begin{equation*}
\sum_{i=1}^{m}\sum_{j=1}^{n}a_{ijij}<-\sum_{i,k=1\atop i\neq k}^{m}\sum_{j,l=1\atop j\neq l}^{n}a_{ijkl}.
\end{equation*}
\end{corollary}

Finally, we introduce a corollary and apply it to construct a nonsingular elasticity M-tensor of the rhombic system.
\begin{corollary}\label{random-1}
Let $\mathcal B=(b_{ijkl})$ be a nonnegative elasticity tensor, and $\mathcal A=R_{1}(\mathcal B)\mathcal I-\mathcal B$ or $\mathcal A=R_{2}(\mathcal B)\mathcal I-\mathcal B$. Then $\mathcal A$ is an elasticity M-tensor.
\end{corollary}
\begin{proof}
By Theorem \ref{Theorem-set1}, we have $R_{1}(\mathcal B)\geq \rho_{M}(\mathcal B)$ and $R_{2}(\mathcal B)\geq \rho_{M}(\mathcal B)$. Recall the Definition \ref{defM-tensor}, the proof is completed.
\end{proof}

The strong ellipticity of the rhombic system has been discussed by many researchers \cite{chirictua2007strong,han2009conditions}. We know from \cite{han2009conditions} that the elasticity tensor $\mathcal A$ of rhombic elastic materials has the following property:
\begin{align}\label{rhombicsystem}
\begin{aligned}
a_{2331} = a_{3331} = a_{2312} =a_{3323} = a_{3112} = a_{3312} =  0,\\
a_{1112} = a_{1123} = a_{2212}= a_{1131}= a_{2231}= a_{2223}=0.
\end{aligned}
\end{align}

\textbf{Example 5.2}~Consider the nonnegative elasticity tensor $\mathcal B=(b_{ijkl})\in\mathbb R^{3\times 3\times 3\times 3}$ that satisfies the rhombic system (\ref{rhombicsystem}), without loss of generality, whose nonzero entries are
$$b_{1111}=1,~~b_{2222}=2,~~b_{3333}=3,~~b_{1122}=4,~~b_{2233}=5,$$
$$b_{3311}=6,~~b_{2323}=7,~~b_{1313}=8,~~b_{1212}=9.$$

By Theorem \ref{Theorem-set1}, we have $R_{2}(\mathcal B)=12.7757$. By Corollary \ref{random-1}, $12.7757\mathcal I-\mathcal B$ is an elasticity M-tensor. Let $\mathcal A=13\mathcal I-\mathcal B$, obviously, $\mathcal A$ also satisfies the rhombic system (\ref{rhombicsystem}). Since $b_{2121}=0$, by Theorem \ref{suff-M}, we know that this elasticity tensor $\mathcal A$ is a nonsingular elasticity M-tensor with the strong ellipticity condition.

\input{ref}
\end{document}

%% file: ex_shared.tex
\usepackage{lipsum}
\usepackage{amsfonts}
\usepackage{graphicx}
\usepackage{epstopdf}
\usepackage{algorithmic}
\usepackage{tikz}
\usepackage{subfigure}
\usepackage{braket}
\ifpdf
  \DeclareGraphicsExtensions{.eps,.pdf,.png,.jpg}
\else
  \DeclareGraphicsExtensions{.eps}
\fi

\newcommand{\uinner}[2]{\ensuremath{\langle #1 ,\; #2 \rangle}}
\newsiamremark{remark}{Remark}
\newsiamremark{hypothesis}{Hypothesis}
\crefname{hypothesis}{Hypothesis}{Hypotheses}
\newsiamthm{claim}{Claim}
\makeatother
\def\xx{{\bf x}}
\def\yy{{\bf y}}

\headers{Tighter Bound Estimation for Efficient Biquadratic Optimization }{S. Li, L. Lu, X. Qiu, Z. Chen, and D. Zeng}
\title{Tighter Bound Estimation for Efficient Biquadratic Optimization Over Unit Spheres\thanks{Submitted to the editors DATE.
\funding{This work is supported by National Natural Science Foundations of China (Grant No. 61571005; Grant No. 12061025; Grant No. 12161020), by Science and Technology Foundation of Guizhou Province ([2020]1Z002), and by the Fundamental Research Program of Guangdong, China (No. 2020B1515310023).}}}
\author{Shigui Li\thanks{School~of~Mathematics,~South China University of Technology, Guangzhou, China
(\email{lishigui@mail.scut.edu.cn}).}
\and Linzhang Lu\thanks{School of Mathematical Sciences, Guizhou Normal University, Guiyang, China; School of Mathematical Sciences, Xiamen University, Xiamen, China (\email{lzlu@xmu.edu.cn}).}
\and Xing Qiu\thanks{Department of Biostatistics and Computational Biology, University of Rochester, Rochester, New York, U.S.A. (\email{xing\_qiu@urmc.rochester.edu}).}
\and Zhen Chen\thanks{School of Mathematical Sciences, Guizhou Normal University, Guiyang, China (\email{zchen@gznu.edu.cn}).}
\and Delu Zeng\footnotemark[1] \thanks{Corresponding author: School of Electronic and Information Engineering, South China University of Technology, Guangzhou, China (\email{dlzeng@scut.edu.cn}).}
}

\usepackage{amsopn}
